\documentclass[leqno]{article}

\usepackage[ruled,vlined]{algorithm2e}

\usepackage[letterpaper]{geometry}

\usepackage{siamproceedings}

\usepackage[T1]{fontenc}
\usepackage{bbm}
\usepackage{amsfonts}
\usepackage{graphicx}
\usepackage{epstopdf}
\usepackage{enumitem}
\usepackage{algorithmic}
\usepackage{booktabs}
\ifpdf
  \DeclareGraphicsExtensions{.eps,.pdf,.png,.jpg}
\else
  \DeclareGraphicsExtensions{.eps}
\fi

\newsiamremark{remark}{Remark}
\newsiamremark{hypothesis}{Hypothesis}
\crefname{hypothesis}{Hypothesis}{Hypotheses}
\newsiamthm{claim}{Claim}

\newtheorem{assumption}[theorem]{Assumption}

\newcommand{\R}{\mathbb{R}}
\usepackage{xcolor}

\definecolor{darkgreen}{rgb}{0,0.7,0}

\usepackage{amsopn}

\footnotetext[2]{Department of Applied Mathematics, University of Colorado, Boulder}
\footnotetext[3]{School of Industrial and Systems Engineering, Georgia Tech}%katya.scheinberg@isye.gatech.edu}
\begin{document}

\newcommand\relatedversion{}

\title{Powell-Style Model-Based Derivative-Free Optimization with Complexity Guarantees}
\author{A. Chaudhry\footnotemark[2]   \and K. Scheinberg\footnotemark[3] \and Scholar Sun\footnotemark[3]}

\date{}

\maketitle

% Copyright Statement
% When submitting your final paper to a SIAM proceedings, it is requested that you include
% the appropriate copyright in the footer of the paper.  The copyright added should be
% consistent with the copyright selected on the copyright form submitted with the paper.
% Please note that "20XX" should be changed to the year of the meeting.

% Default Copyright Statement
%\fancyfoot[R]{\scriptsize{Copyright \textcopyright\ 20XX by SIAM\\
%Unauthorized reproduction of this article is prohibited}}

% Depending on which copyright you agree to when you sign the copyright form, the copyright
% can be changed to one of the following after commenting out the default copyright statement
% above.

%\fancyfoot[R]{\scriptsize{Copyright \textcopyright\ 20XX\\
%Copyright for this paper is retained by authors}}

%\fancyfoot[R]{\scriptsize{Copyright \textcopyright\ 20XX\\
%Copyright retained by principal author's organization}}

% \pagenumbering{arabic}
% \setcounter{page}{1}%Leave this line commented out.
\thispagestyle{plain}
\pagenumbering{arabic}
\setcounter{page}{1}
\pagestyle{plain}

\begin{abstract}We propose variants of model-based trust region derivative free algorithms that are closest to methods initially proposed and implemented by Powell in \cite{UOBYQA,MJDPowell_2004b}. These methods rely on low degree polynomial interpolation and   carefully maintain  geometry of the interpolation sets. We are able to derive  complexity bounds for these methods that make them theoretically competitive to other derivative free methods. Applying these methods in randomly generated subspaces
recovers what we believe to be nearly tight complexity. This paper builds on recent results in \cite{ChaudhryScheinberg2026ICM} where complexity of a much simplified version of Powell's methods was derived. Here we extend the analysis to fully incorporate Powell's geometry handling approach, 
and conduct extensive numerical comparison of the model-based trust region methods connecting practical and theoretical performance. 
We also extend the analysis of subspace model-based trust region methods initially developed in  \cite{ChaudhryScheinberg2026ICM} to the case of noisy function 
evaluations.

 %MSC Classification: 90C30, 90C56
\end{abstract}
\section{Introduction.}
 In this paper, we focus on the complexity of model-based derivative free algorithms that aim to solve unconstrained nonlinear optimization problems of the form
\begin{align} \label{eq.prob}
	\min_{x \in \mathbb{R}^n} \phi(x), 
\end{align}
where $\phi: \mathbb{R}^n \rightarrow \mathbb{R}$ is a smooth objective function that may not be convex. The key premise of these methods is to economize on function values, possibly at the expense of additional linear algebra, since in most applications the function evaluations cost dominates all else. The complexity will be measured in terms of the number of function evaluations needed to achieve an $\epsilon$-stationary point, that is a point $x$ for which $\|\nabla \phi(x)\|\leq \epsilon$. 

The main goal of this paper is to narrow the gap between theory and practice in model based   derivative free optimization (DFO). Derivative free optimization, also known under names    {\em zeroth-order} and  {\em gradient-free} optimization,
is an area of optimization  concerned with developing optimization methods based purely on function value computation without applying any direct differentiation. The area has experienced significant growth both in theory and in a variety of applications in the past several decades. There is a rich literature of DFO starting from around mid 90s which is rapidly growing with new interest spurred by new applications in engineering, machine learning and artificial intelligence. 
 Aside from an increasing number of papers, there are two books \cite{DFOBook} and \cite{audet2017derivativefree}, and a survey on the topic \cite{LarsMeniWild2019}. Examples of applications can be found in \cite{audet2017derivativefree} and \cite{rios2013derivativefree}.

 Model-based DFO methods approach the optimization problem by constructing and maintaining (usually local) models of the objective function from function value samples. A particular class of such methods, model-based  trust region methods pioneered in the 90s by M.J.D. Powell \cite{UOBYQA,MJDPowell_2004b}, proved to be very 
 effective in practice for many applications \cite{more2009benchmarking}. 
 These methods use polynomial interpolation models and maintain sample sets using carefully engineered techniques, supported by mathematical properties of Lagrange interpolation polynomials.  The complex structure of the algorithms, especially as presented in Powell's papers, made them difficult to implement, let alone to analyze rigorously.  Some underlying theory of asymptotic convergence was developed in \cite{DFOBook} and later complexity bounds were provided for those  algorithms in  \cite{garmanjani2016trust}, however, methods analyzed there  are much simplified and include elements that significantly depart from  Powell's ideas.   More specifically, Powell's methods only use one or two function evaluations per iteration, attempting to select the sample points in an optimal way, to improve the objective value while also maintaining or improving  geometry of sample points. In other words, they carefully balance exploration and exploitation. On the other hand,  the algorithms which have enjoyed complexity bounds so far require occasional complete model reconstruction, which requires far more function evaluations and essentially abandon Powell's careful geometry correcting approach. 
 
 For the past several decades there was a general lack  of understanding of how Powell's methods work and most importantly if they enjoy favorable complexity bounds. Recently, in a considerable implementational effort Powell's software has been reincarnated in modern platforms by Z. Zhang and his colleagues \cite{pdfo, prima}. Their software packages have seen considerable success and are now being widely used by practitioners. This raises the interest in providing solid theory for these methods. In \cite{ChaudhryScheinberg2026ICM}  a {\em geometry correcting} method inspired by Powell's algorithms was proposed and its complexity bounds derived. There were several important questions addressed in that paper. 
 \begin{itemize}
 \item It was shown that a model-based TR method which performs only one or two function evaluations per iteration has good complexity guarantees in the case of general polynomials. 
 \item In the case of linear polynomial interpolation the worst case complexity bound is ${\cal O}(n^2\epsilon^{-2})$ which matches those of other DFO methods, such as direct search or methods based on finite difference gradient approximation. 
 \item It was shown that  model-based trust region methods (Powell's or other) can be applied  within a  random subspace algorithm
 which results in a further  improvement  in the worst case complexity bound to ${\cal O}(n\epsilon^{-2})$. 
 \end{itemize}
 
 In this paper we extend the results of \cite{ChaudhryScheinberg2026ICM} in several important ways. 
 \begin{itemize}
 \item The method in  \cite{ChaudhryScheinberg2026ICM}  still departs from Powell's method in its geometry handling approach. We will explain the details of this when we describe the algorithm but the key difference is that the method in  \cite{ChaudhryScheinberg2026ICM} discards some potentially useful sample points, which Powell's methods do not. Including these points complicates the theory substantially, but we are able to provide such theory here. 
 \item We derive the theory based on linear interpolation since this is sufficient for the first-order convergence analysis, however, with an easy modification we allow for quadratic models without any additional complexity cost. 
 \item We extend the analysis to functions with deterministic noise, which is essential when addressing real DFO applications. Deterministic noise implies a lower bound on the best reachable optimality criterion and we derive such a  lower bound. 
 \item We extend the analysis of the random subspace DFO trust region method to the case of deterministic noise. As we will show, there is an additional complication which requires algorithmic modifications in this case (as opposed to the full-space case or noise-free subspace case). 
 \item Finally, we provide a careful implementation of our ideas and demonstrate that our theoretically supported algorithm can match the performance of  other trust region DFO methods, including those in \cite{pdfo,prima}. 
\end{itemize}
 
\subsection{Preliminaries}

The following are the standard assumptions on $\phi(x)$ for our setting. 
\begin{assumption}[\textbf{Lower bound on $\pmb{\phi}$}] \label{assum:low_bound} 
	The function $\phi$ is bounded below by a scalar $\phi^\star$ on $\mathbb{R}^n$. 	
\end{assumption}

\begin{assumption}[\textbf{Lipschitz continuous gradient}]	\label{assum:lip_cont}  
	The function $\phi$ is continuously differentiable, and the gradient of $\phi$ is $L$-Lipschitz continuous on $\mathbb{R}^n$, i.e., $\|\nabla \phi(y) - \nabla \phi(x)\| \le L \|y-x\|$ for all $(y,x) \in \mathbb{R}^n\times\mathbb{R}^n$.
\end{assumption}

 Throughout the paper, we assume that we do not have access to $\nabla \phi(x)$ or its approximation using any form of differentiation. Instead we  have access to an inexact zeroth-order oracle $f(x)\approx {\phi}(x)$. Specifically, for all $x$ we assume
\[
|f(x)-\phi(x)|\leq \epsilon_f
\]
for some $\epsilon_f>0$. 

Our main objective is to derive algorithms with the following type of guarantees.  For a given $\epsilon>0$, let ${\cal C}_\epsilon$ denote  the total number of calls to oracle $f(x)$ performed by the given algorithm that guarantees reaching a point $x_\epsilon$ for which  $\|\nabla \phi(x_\epsilon)\|\leq \epsilon$. Then  we seek the following guarantees:
$ \forall \epsilon>\psi(\epsilon_f)$ ${\cal C}_\epsilon\leq \Psi(n, \epsilon)$, if the algorithm is deterministic,  or 
${\mathbb E}[{\cal C}_\epsilon]\leq \Psi(n, \epsilon)$, if the algorithm is stochastic. Here $\psi$ is some function of the oracle noise that gives the best achievable optimality criterion  and  $\Psi$ is the complexity bound. We will make use of ${\cal O}()$ notation to suppress dependence on constants in upper bounds, $\Omega()$ notation to do the same in the case of lower bounds and $\Theta()$ to be used  with an equality up to a constant factor.  For all methods considered here as for all methods of similar type  $\Psi()={\cal O}(\epsilon^{-2})$. Thus the main focus of this work  is in the dependence of $\Psi(n, \epsilon)$ on $n$ and the dependence of  $\psi()$ on  $n$ and $\epsilon_f$. This reflects practical concerns with respect to derivative-free optimization, which in contrast to the usual derivative-based optimization has complexity dependence on $n$ and is known not to scale well for large dimensional problems. 

The paper is organized as follows: In Section \ref{sec:complexity} we present the analysis of the basic 
model-based trust region framework and present some key results based on what is known as fully-linear models
 that drive complexity analysis. In Section \ref{sec:lagrange} we establish how these fully linear models arise in polynomial interpolation and quantify their properties in terms of Lagrange polynomials associated with the sample sets. 
In Section \ref{sec:powell} we propose our main algorithm that uses Lagrange polynomials as a tool for model maintenance and analyze its complexity. 
Section \ref{sec:subspace_ffd} focuses on the subspace trust region method with inexact zeroth-order oracle. 
Finally, in Section \ref{sec:numex} we describe implementational enhancements and present our computational comparisons.

%There are three major classes of DFO algorithms: 1. {\bf directional direct search} methods, 2.  
%gradient descent  based on {\bf simplex gradients} which includes finite difference approximation and 3. {\bf interpolation based (also known as model based) trust region methods}. The first two classes are rather simple to implement and analyze which makes them popular choices, especially in the literature and with recent applications to machine learning where they are easy to adopt. They operate by sampling a certain number of  function values at some chosen points around the current iterate and then make a step according to the information obtained from the samples. The way the sample sets are constructed is predetermined (but can be random) and is meant to ensure that a descent direction is identified either deterministically or in expectation. Such  method is then analyzed based on this "guaranteed descent"  and the effort it takes to obtain it. 

%\section{Basic trust-region method and analysis.}
%\label{sec:ffd}
%\input{basic_ffd}

\section{Basic algorithm and elements of complexity analysis.}
\label{sec:complexity}
We first present and analyze the  trust region framework with the focus on  its key elements. 

As in most trust region algorithms,  at every iteration $k \in \{0,1,\dots\}$, 
we construct a quadratic model to approximate $\phi(x)$ near the iterate $x_k$
\begin{equation}\label{eq:model_def}
	m_k(x_k+s) = {\phi}(x_k) +  g_k(x_k)^Ts  + \frac{1}{2} s^T H_k(x_k) s. 
\end{equation}
The model is then minimized (approximately) over the trust region $B(x_k,\Delta_k)$ - a Euclidean ball around $x_k$ of a radius $\Delta_k$. 
In the paper we use the abbreviation $g_k:=g(x_k)= \nabla m_k(x_k)$ and $H_k:= H(x_k)=\nabla^2 m_k(x_k)$.\footnote{Note that the constant term ${\phi}(x_k)$ appears in the definition \eqref{eq:model_def} but will not be needed in the algorithm, since only changes in the model value 
$m_k(x_k)-m_k(x_k+s)$  are of interest.}

The following definition, also widely used in the literature  \cite{DFOBook}, helps us  identify the requirements on models $m_k$  that are critical for convergence.

\begin{definition}[Fully-linear model] \label{def:fully-linear}
	Given a ball  around point $x$ of radius $\Delta$, $B(x,\Delta)$,  we say that model $m(x+s)$ is a $\kappa_{ef}, \kappa_{eg}$-fully linear model of $\phi(x+s)$  on 
	$B(x, \Delta)$ if
	\[
	\|\nabla m(x)-\nabla \phi(x)\|\leq \kappa_{eg}\Delta
	\]
	and 
	\[
	|m(x+s)-\phi(x+s)|\leq \kappa_{ef}\Delta^2
	\]
	for all $\|s\|\leq \Delta$. 
\end{definition}

We now state the algorithmic framework where at each iteration updates are made based on progress and also on whether or not the model is known to be fully linear. We do not need to specify $\kappa_{eg}$ and $\kappa_{ef}$ constants in the algorithm but we assume they exist and are fixed throughout the iterations. In the framework below we do not even discuss how to verify if a model is $\kappa_{ef}, \kappa_{eg}$-fully-linear model, we simply assume such mechanism is given as an input. In the later section we will address this aspect in detail and incorporate it into the algorithm. 
\begin{algorithm}[H] 
	\caption{~\textbf{Trust region method based on fully-linear models}}
	\label{alg:tr}
	{\bf Inputs:} A zeroth-order oracle $f(x)\approx {\phi}(x)$, a starting point $x_0$, TR radius $\Delta_0$, and hyperparameters $\eta_1\in(0,1)$, $\eta_2 > 0$, and $\gamma\in(0,1)$.  Mechanism for establishing if a model is  $\kappa_{ef}, \kappa_{eg}$-fully-linear for some fixed $\kappa_{ef}, \kappa_{eg}$. \\
	\For{$k=0,1,2,\cdots$}{    
		\nl Compute model $m_k$ and a trial step $x_k+s_k$ where $s_k\approx \arg\min_s \{m_k(x_k+s):~s \in B(0,\Delta_k)\}$.\\
		\nl Compute the ratio $\rho_k$ as 
		\[ \rho_k= \frac{{f}(x_k) - {f}(x_k+s_k)}{m_k(x_k) - m_k(x_k+s_k)}.
		\] 
		\nl Update the iterate and the TR radius as 
		\[  (x_{k+1}, \Delta_{k+1}) \gets \left\{ \begin{aligned} 
			&(x_k+s_k, \gamma^{-1} \Delta_k) &&\text{if } \rho_k \ge \eta_1 \text{ and } \|g_k\| \ge \eta_2 \Delta_k, \\
			&(x_k, \Delta_k) && \text{else, if model is not FL in } B(x_k,\Delta_k).  \\
			&(x_k, \gamma \Delta_k) &&\text{otherwise. }
		\end{aligned} \right. 
		\]
		\nl Perform some model improvement steps if appropriate. 
	}
\end{algorithm}

We will make the following standard assumption on the models $m_k$ and their minimization.

\begin{assumption} \label{assum:tr}
	\begin{enumerate}
		\item The trust region subproblem is solved sufficiently accurately in each iteration $k$ so that $x_k+s_k$ provides at least a fraction of Cauchy decrease, i.e. for some constant $0 < \kappa_{fcd} < 1$ 
		\begin{equation}\label{eq:Cauchy decrease}
			m_k(x_k) - m_k(x_k+s_k) \ge \frac{\kappa_{fcd}}{2} \|g_k\| \min\bigg\{\frac{\|g_k\|}{\|H_k\|}, \Delta_k\bigg\}. 
		\end{equation}
		\item There exists a constant $\kappa_{bhm} > 0$ such that, for all $x_k$ generated by Algorithm~\ref{alg:tr}, the spectral norm of the Hessian of the model is bounded as
		\[ \|H_k \| \le \kappa_{bhm}. \] 
	\end{enumerate}
\end{assumption}

Condition \eqref{eq:Cauchy decrease} is commonly used in the literature and is satisfied by the Cauchy point with $\kappa_{\rm fcd} = 1$. 
See \cite[Section 6.3.2]{TRbook} for more details. 

Algorithm \ref{alg:tr} is a variant of a  standard trust region method, well studied in the literature \cite{TRbook}. There are two key differences that appear in trust region methods specifically in the DFO context. The first one is the   condition $\|g_k\| \ge \eta_2 \Delta_k$. This condition is not used in classical TR methods where $\nabla m_k(x_k)=g_k=\nabla \phi(x_k)$ and was first introduced in \cite{bandeira2014convergence} for the case of a TR method based on  random models. The reason for this condition is tied to the fact that the trust region radii have a dual function in this setting - controlling the step of the algorithm and also controlling the model accuracy. For the same reasons, most  prior deterministic DFO literature relies on a much less practical and cumbersome "criticality step". The condition $\|g_k\| \ge \eta_2 \Delta_k$   ensures that the trust region radius and thus the gradient accuracy stays on track as the norm of the gradient reduces which removes the need for a separate  criticality step. The second difference for the derivative-free model-based TR  (versus the classical TR) method is the necessity to ensure that the model is fully-linear in the trust region when the trust region radius is reduced. One can simply assume that this is guaranteed, for example, by using finite difference gradient approximation to build the model at each iteration. This gives an easy and convenient way to analyze the algorithm, but this does not result in a practical method. In various prior works \cite{UOBYQA, DFOBook, ChaudhryScheinberg2026ICM} efficient methods have been proposed to recognize whether a model is fully-linear and if not, to make an improving step. The purpose of this paper is to improve on these methods both in terms of theory and practice.

Step 3 of Algorithm \ref{alg:tr} identifies three types of iterations: {\em successful} iteration, where the trial step is accepted and the trust region radius is increased, {\em unsuccessful} iteration, where the step is rejected and the trust region radius is decreased and what we will call {\em model improving } iteration where the step is rejected but the trust region radius is not decreased because the model is not fully-linear. 

In this section we will derive the bound on the total number of successful and unsuccessful iterations and in the following section we will provide a mechanism for model improvement and will bound the number of the model improving iterations.  For a given $\epsilon>0$, let $K_\epsilon$ be the first iteration of Algorithms \ref{alg:tr} for which $\|\nabla \phi(x_k)\| \le \epsilon$.  We define the index sets
\begin{align} \label{eq:smu_def}
	{\cal S}_{\epsilon} &:= \{k\in\{0,\dots,K_\epsilon-1\}:~ \text{iteration $k$ is successful}\}. \nonumber\\
	{\cal M}_{\epsilon} &:= \{k\in\{0,\dots,K_\epsilon-1\}:~ \text{iteration $k$ is model improving}\}.\\
	{\cal U}_{\epsilon} &:= \{k\in\{0,\dots,K_\epsilon-1\}:~ \text{iteration $k$ is unsuccessful}\nonumber \}.
\end{align}

We next present two lemmas that are  key in the analysis of any trust region algorithm and specifically Algorithm~\ref{alg:tr}. The first lemma establishes that once $\Delta_k$ is sufficiently small compared to the gradient norm, a successful step is ensured. 
\begin{lemma}[small $\Delta_k$ implies successful step] \label{lem:tr_success}
	Under Assumption~\ref{assum:tr}, if $m_k$ is $\kappa_{ef}, \kappa_{eg}$-fully-linear and 
	\begin{equation} \label{eq:tr_success_ef}
	\sqrt{\frac{2\epsilon_f}{C_0}}\le 	\Delta_k \le C_1 \|\nabla \phi(x_k)\| \text{, where } C_1=\bigg(\max\left\{\eta_2,\ \kappa_{bhm},\ \frac{2\kappa_{ef}+C_0}{(1-\eta_1)\kappa_{fcd}}\right\}+\kappa_{eg}\bigg)^{-1}, 
	\end{equation}
	and $C_0$ is an arbitrary constant which we pick to equal $\max \{\eta_2, \kappa_{ef}\}$ for future convenience, then $\rho \ge \eta_1$, $\|g_k\| \ge \eta_2 \Delta_k$, thus iteration $k$ is successful and  $x_{k+1} = x_k+s_k$. 
\end{lemma}
The  proof is a simple modification of those in the stochastic trust region literature  \cite{blanchet2019convergence, cao2023first}\footnote{This lemma was stated erroneously without proof in \cite{ChaudhryScheinberg2026ICM} with $4 \epsilon_f$ instead of $C_0$. The error is inconsequential, but we correct it here.}

\begin{proof}
	By the assumption that $m_k$ is fully-linear, we have  
	\[ \|\nabla {\phi}(x_k)\| \le \|g_k\| + \kappa_{eg}\Delta_k. 
	\]
	From \eqref{eq:tr_success_ef} we conclude that 
	\[ \begin{aligned}
		(\max\{\kappa_{bhm},\eta_2\} + \kappa_{eg}) \Delta_k &\le \|\nabla {\phi}(x_k)\| \le \|g_k\| + \kappa_{eg}\Delta_k \\
		\max\{\kappa_{bhm},\eta_2\} \Delta_k &\le \|g_k\|. 
	\end{aligned} \] 
	This implies that the first condition of the successful step, namely $\|g_k\| \ge \eta_2\Delta_k$, is satisfied. 
	From \eqref{eq:Cauchy decrease} we have 
	$m_k(x_k) - m_k(x_k+s_k) \ge \kappa_{fcd} \|g_k\| \Delta_k / 2$. 
	Thus, using  the assumption that $m_k$ is fully-linear again,
	\[ \begin{aligned}
		\rho_k &= \frac{m(x_k) - m(x_k+s_k) + ({f}(x_k) - m(x_k)) - ({f}(x_k+s_k) - m(x_k+s_k))}{m(x_k) - m(x_k+s_k)} \\
		& = \frac{m(x_k) - m(x_k+s_k) + ({\phi}(x_k) - m(x_k)) - ({\phi}(x_k+s_k) - m(x_k+s_k))}{m(x_k) - m(x_k+s_k)} \\ &+ \frac{f(x_k) -\phi(x_k) + ({f}(x_k+s_k) - \phi(x_k+s_k))}{m(x_k) - m(x_k+s_k)} \\
		&\ge 1 - \frac{\kappa_{ef} \Delta_k^2}{m(x_k) - m(x_k+s_k)} - \frac{2\epsilon_f}{m(x_k) - m(x_k+s_k)} \\
		&\ge 1 - \frac{\kappa_{ef} \Delta_k^2}{\kappa_{fcd} \|g_k\| \Delta_k / 2} -  \frac{2\epsilon_f}{\kappa_{fcd} \|g_k\| \Delta_k / 2} \ge 1 - \frac{(2\kappa_{ef} +2\epsilon_f/\Delta_k^2)\Delta_k}{\kappa_{fcd} (\|\nabla {\phi}(x_k)\| - \kappa_{eg}\Delta_k)} \\
		&\ge 1 - \frac{(2\kappa_{ef} +C_0)\Delta_k}{\kappa_{fcd} (\|\nabla {\phi}(x_k)\| - \kappa_{eg}\Delta_k)}\ge \eta_1, 
	\end{aligned} \]
	where the last inequality  follows from \eqref{eq:tr_success_ef} since $\|\nabla {\phi}(x_k)\| \ge \big(\frac{2\kappa_{ef}}{(1-\eta_1)\kappa_{fcd}} + \kappa_{eg}\big) \Delta_k$.  
\end{proof}

\begin{lemma}[successful iteration implies function reduction] \label{lem:tr_progress}
	Let Assumptions~\ref{assum:lip_cont} and \ref{assum:tr} hold. 
	If  the iteration $k$ is successful, then   
	   \begin{equation}\label{eq:tr_progress_ef}
   \phi(x_k) - \phi(x_{k+1}) \ge C_2\Delta_k^2 -2\epsilon_f \text{, where } C_2 = \frac{\eta_1\eta_2\kappa_{fcd}}{2} \min\{\frac{\eta_2}{\kappa_{bhm}}, 1\}; 
    \end{equation}
	otherwise, we have $x_{k+1} = x_k$ and ${\phi}(x_k) - {\phi}(x_{k+1}) = 0$.  
\end{lemma}
This  proof   can be found in 
\cite{ChaudhryScheinberg2026ICM}. 
    Thus, under the additional assumption that $\Delta_k\geq \sqrt{\frac{2\epsilon_f}{\tau C_2}} $ for some $\tau\in(0,1)$ \eqref{eq:tr_progress_ef} can be further stated as 
     \begin{equation}\label{eq:tr_progress_tau}
   \phi(x_k) - \phi(x_{k+1}) \ge (1-\tau) C_2\Delta_k^2 \text{, where } C_2 =  \frac{\eta_1\eta_2\kappa_{fcd}}{2} \min\{\frac{\eta_2}{\kappa_{bhm}}, 1\}.
    \end{equation}
    
        Lemmas \ref{lem:tr_success} and \ref{lem:tr_progress} give us a bound on the total number of successful and unsuccessful iterations, 
        as long as it can be ensured that    $\Delta_k$ remains 
   not smaller than $\sqrt{\frac{2\epsilon_f}{\min\{\tau C_2, C_0\}}}$. Henceforth, for simplicity we use $\tau=\frac{1}{2}$, also since we chose $C_0\geq \eta_2$ and $\eta_2 \geq C_2$  the bound on $\Delta_k$ reduces to $\sqrt{\frac{4\epsilon_f}{ C_2}}$. Due to Lemma \ref{lem:tr_success} and the update mechanism for $\Delta_k$ the bound  is ensured as long as $
   \|\nabla \phi(x_k)\|\geq \epsilon$ for $\epsilon$ sufficiently large.

\begin{lemma}[Lower bound on $\Delta_k$]\label{lem:delta_bnd}
	For any $\epsilon>  \sqrt{\frac{4\epsilon_f}{\gamma^2  C_2 C_1^2}}$,  assuming that $\Delta_0 > \gamma C_1 \epsilon$, 
	for all $k \in \{0,\dots, K_\epsilon - 1\}$ $\Delta_k \ge \gamma C_1 \epsilon$.
\end{lemma}
\begin{proof}
	According to Lemma~\ref{lem:tr_success}, any iteration $k \in \{0,\dots, K_\epsilon - 1\}$ must be successful when $\Delta_k \le C_1\epsilon$ and $m_k$ is $\kappa_{ef}, \kappa_{eg}$-fully linear. Thus no iteration can be unsuccessful when  $\Delta_k \le C_1\epsilon$. 
	Thus, given $\Delta_0 > \gamma C_1 \epsilon$, and by the mechanism of Algorithm \ref{alg:tr} we must have $\Delta_k \ge \gamma C_1 \epsilon$ for all $k \in \{0,\dots, K_\epsilon - 1\}$.
\end{proof}  

The following bound holds under the result of Lemma \ref{lem:delta_bnd}. 
\begin{lemma}[Bound of successful iterations]\label{lem:succ_iter_bnd}
	For any $\epsilon>  \sqrt{\frac{4\epsilon_f}{\gamma^2  C_2 C_1^2}}$, assuming $\Delta_0 > \gamma C_1 \epsilon$,   we have
	\[
	|{\cal S}_\epsilon| \leq \frac{2(\phi(x_0)-\phi^\star)}{C_2 (\gamma C_1 \epsilon)^2 }
	\]
\end{lemma}
\begin{proof}
	Using Lemma~\ref{lem:tr_progress} with $\tau=\frac{1}{2}$ and from $\Delta_k \ge \gamma C_1 \epsilon$ for all $k \in \{0,\dots, K_\epsilon - 1\}$ we have  
	\begin{align*} 
		\phi(x_0) - \phi^\star &\ge \sum_{k=0}^{K_\epsilon-1} {\phi}(x_k) - {\phi}(x_{k+1}) \ge \frac{1}{2} \sum_{k\in{\cal S}_\epsilon} C_2 \Delta_k^2 > \frac{1}{2} |{\cal S}_\epsilon | C_2 (\gamma C_1 \epsilon)^2 
		%\\ 
		%    &\ge \lceil{\Big( T_\epsilon -  \log_\gamma \frac{C_1\epsilon}{\delta_0} \Big) / 2} \rceil C_2 (\gamma C_1 \epsilon)^2. 
	\end{align*}
	which gives the result of the lemma.
\end{proof}

\begin{lemma}\label{lem:unsucc_iter_bnd}
	For any $\epsilon>  \sqrt{\frac{4\epsilon_f}{\gamma^2  C_2 C_1^2}}$, assuming that the initial trust-region radius $\Delta_0 > \gamma C_1 \epsilon$, 
	\[
	|{\cal U}_\epsilon| \le  |{\cal S}_\epsilon| + \lceil{\Big(  \log_\gamma \frac{C_1\epsilon}{\Delta_0} \Big)}\rceil. 
	\]
\end{lemma}
\begin{proof}
	We observe that 
	\[ 
	\Delta_{K_\epsilon} =  \gamma^{-|{\cal S}_\epsilon|} \gamma^{|{\cal U}_\epsilon|} \Delta_0  \geq   C_1 \epsilon, 
	\]  
	The last inequality follows from the fact that $ \Delta_{K_\epsilon-1}\geq  \gamma C_1 \epsilon$ and the $K_\epsilon-1$-th iteration must be successful. 
	Thus the number of unsuccessful iterations can be bounded using the number of successful ones rearranging the terms and  taking the logarithm. 
\end{proof}

\begin{theorem}\label{thm:tr_complexity_ef}
	Let Assumptions~\ref{assum:lip_cont} and~\ref{assum:tr}  hold. 
	For any $\epsilon>  \sqrt{\frac{4\epsilon_f}{\gamma^2  C_2 C_1^2}}$, assuming that the initial trust-region radius $\Delta_0 > \gamma  C_1 \epsilon$, where $ C_1$ is defined as $\bigg(\max\left\{\eta_2,\ \kappa_{bhm},\ \frac{2\kappa_{ef}+\max\{\eta_2, \kappa_{ef}\}}{(1-\eta_1)\kappa_{fcd}}\right\}+\kappa_{eg}\bigg)^{-1}$,
	then we have the bound
	\begin{equation} \begin{aligned} 
			|{\cal S}_\epsilon|+ |{\cal U}_\epsilon|\le \frac{4(\phi(x_0) - \phi^\star)}{C_2 (\gamma  C_1 \epsilon)^2} +  \lceil{\Big(  \log_\gamma \frac{C_1\epsilon}{\Delta_0} \Big)}\rceil.
		\end{aligned} \end{equation}
\end{theorem}

Constants $C_1$ and $C_2$ have a direct effect on the complexity and we will be using them (and their variations) throughout the paper. 
Let us  pause here to understand their different components. Specifically, constants $\gamma$, $\eta_1$ and $\kappa_{fcd}$ are usually chosen to be fixed in the algorithm and be close to $1$ (say $0.9$).
% $\tau$ is simply making sure that $\epsilon$ is separated from the lowest allowable value, thus is too can be a constant  smaller than $1$. In what follows we will simply choose $\tau=\frac{1}{2}$. 

The key remaining constants are  $\eta_2$ which is chosen in the algorithm and is used as specified, and 
$\kappa_{ef}$, $\kappa_{eg}$  and $\kappa_{bhm}$ which are all attributes of the constructed models and are not specified by the algorithm but are rather upper bounded by theory. 

In what follows we will impose an upper bound on $\kappa_{bhm}$ which will be dimension independent. Ideally $\kappa_{bhm}$ should scale similarly to $L$, since the later is the bound on the norm of the true Hessian and the former is the bound on the model Hessian. 
Thus it is convenient to think of $\kappa_{bhm}$ as $\sim {\cal O}(L)$, although it also can be zero if linear models are used but also can be large if allowed. 

What remains is to derive  concrete bounds on $\kappa_{eg},\kappa_{ef}$ with explicit dependence on dimension. These would depend on the manner in which the models are constructed. 

The following standard lemma shows that a bound $\kappa_{eg}$ (together with $\kappa_{bhm}$) implies a bound on $\kappa_{ef}$. 

\begin{lemma}[Fully linear models]\label{lem:fully-lin}
	Under Assumptions~\ref{assum:lip_cont} and  \ref{assum:tr}  if 
	\begin{equation}\label{eq:kappa_eg}
	\|\nabla m({x}) - \nabla {\phi}({x})\|\leq \kappa_{eg}\Delta
	\end{equation}
	 then $m_k(x_k+s)$ is a $\kappa_{ef}, \kappa_{eg}$-fully linear model of $\phi(x+s)$ on $B(x,\Delta)$  with $\kappa_{ef}= \kappa_{eg}+\frac{L+\kappa_{bhm}}{2}$.
\end{lemma}

Let us consider a concrete way of building fully-linear models and the complexity implications. 
From analysis of the finite difference gradient approximation error (see e.g.\cite{berahas2021theoretical}), if one forms a gradient estimate via
\begin{equation}\label{eq:finite_diffs}
	g(x) =  \sum_{i=1}^n \frac{f(x+{\delta} u_i) - f(x)}{{\delta} } u_i, 
\end{equation}
then one can derive a gradient approximation bound of 
\begin{align*}
	\|g({x}) - \nabla {\phi}({x})\| \leq \frac{\sqrt{n} L {\delta}}{2} + \frac{2\sqrt{n}  \epsilon_f}{{\delta} }.
\end{align*}

Taking $\delta=\Delta_k$, at each iteration $k$, results in a
$\kappa_{ef}, \kappa_{eg}$-fully linear model in $B(x_k,\Delta_k)$ with 
$\kappa_{eg}=\sqrt{n}L$ and $\kappa_{ef}=\frac{L+\kappa_{bhm}}{2}+\sqrt{n}L$, as long as $\Delta_k\geq   2\sqrt{\frac{\epsilon_f}{L}}$.
To ensure this  we add the lower bound $\gamma C_1 \epsilon \geq   2\sqrt{\frac{\epsilon_f}{L}}$ which translates to 
$\epsilon\geq   \sqrt{\frac{4\epsilon_f}{\gamma^2 L C_1^2}}$.

The total  bound on the number of function evaluations, i.e., oracle complexity, easily follows from Theorem \ref{thm:tr_complexity_ef} and from the fact that \eqref{eq:finite_diffs} gives a fully linear model at the cost of $n+1$ oracle calls. Thus there are no model improvement iterations and the final bound on the oracle complexity is derived via the bound on $(n+1)(|{\cal S}_\epsilon|+ |{\cal U}_\epsilon|)$. 
As noted in \cite{ChaudhryScheinberg2026ICM}, if $\eta_2$ is taken to be a constant independent of dimension, using this finite difference scheme, the total worse-case oracle complexity to achieve $\|\nabla {\phi}(x_k)\| \leq \epsilon$ for any $\epsilon>  \sqrt{\frac{4\epsilon_f}{\gamma^2 \min\{C_2, L\} C_1^2}}$ is bounded by
\[
{\cal C}_\epsilon \le {\cal O}(n^2\epsilon^{-2}).
\]

In the following corollary, we show that this complexity is not optimal and can, in fact, be improved by taking $\eta_2$ to grow with the dimension of the problem by improving the bound on $|{\cal S}_\epsilon|+ |{\cal U}_\epsilon|$ when $\kappa_{ef}$ and  $\kappa_{eg}$ scale as ${\cal O}(\sqrt{n})$.

\begin{corollary}
	Under the same assumptions as Theorem~\ref{thm:tr_complexity_ef}, assuming that $g_k(x_k)$ is computed  via \eqref{eq:finite_diffs} with $\delta = \Delta_k$ at each iteration,  choosing $\eta_2 = \sqrt{n}$,  and assuming $\kappa_{bhm}\leq {\cal O}(\sqrt{n})$,  then  for any $\epsilon>  \sqrt{\frac{4\epsilon_f}{\gamma^2 \min \{C_2, L\} C_1^2}}$ 
	\[
	|{\cal S}_\epsilon|+ |{\cal U}_\epsilon|\leq {\cal O} \left( \frac{\sqrt{n}}{\epsilon^2} \right)
	\]
	and  the total worst-case oracle complexity is bounded by 
	\[
	{\cal C}_\epsilon\le {\cal O} \left( \frac{n^{3/2}}{\epsilon^2} \right)
	\]
\end{corollary}
\begin{proof}
By the error bound from \cite{berahas2021theoretical}, we have $\kappa_{ef}, \kappa_{eg} = O(\sqrt{n})$.
Recall $C_1^{-1} = \max\left\{\eta_2,\ \kappa_{bhm},\ \frac{2\kappa_{ef}+\max\{\eta_2, \kappa_{ef}\}}{(1-\eta_1)\kappa_{fcd}}\right\}+\kappa_{eg}$ and $C_2 = \frac{\eta_1\eta_2\kappa_{fcd}}{2} \min\left \{\frac{\eta_2}{\kappa_{bhm}}, 1\right \}$.
Thus, under the assumptions of this corollary, we have $C_1^{-1} = \Theta(\sqrt{n})$ and $C_2 = \Theta(\sqrt{n})$.
Thus, we can bound
	\begin{equation} \begin{aligned} 
		|{\cal S}_\epsilon|+ |{\cal U}_\epsilon| &\le \frac{2(f(x_0) - f^\star)}{C_2 (\gamma C_1 \epsilon)^2} + \log_\gamma \frac{C_1\epsilon}{\Delta_0} \\ 
		&=  {\cal O} \left(\frac{\sqrt{n}(\phi(x_0) - \phi^\star)}{\epsilon^2}  + \log \left( \frac{\sqrt{n} \Delta_0}{\epsilon}\right)\right)= {\cal O} \left(\frac{\sqrt{n}}{\epsilon^2}\right).
\end{aligned} \end{equation}
The total oracle complexity follows as discussed before. 
\end{proof}

We note here that since we assume that $L$ does not scale with $n$, while  $C_1^{-1} = \Theta(\sqrt{n})$, then whether we choose $\eta_2$ to be constant or to equal $\sqrt{n}$ the lower bound  on  $\epsilon$   is $\Omega\left (\sqrt{n\epsilon_f}\right )$.

In the next section we show how fully-linear models can be constructed via polynomial interpolation, using less rigid sample sets than used in 
\eqref{eq:finite_diffs}  and yet we are able to derive a competitive bound on   $\kappa_{eg}$ (and thus  $\kappa_{ef}$).

\section{Lagrange polynomials and fully-linear models.}
\label{sec:lagrange}

%The following theorem can be found in \cite{}
%\begin{theorem}
%For $m(x)$ - an interpolation polynomial of degree $d$
%\[
%|\phi(x)-m(x)|
%\leq \frac{L}{(d+1)!} \sum_{j=0}^{p}\|y_j-x\|^{d+1} |\ell_j(x)|,
%\]
%\end{theorem}
Before introducing the method we wish to analyze in this paper we need to discuss an important tool utilized by these algorithms - Lagrange polynomials. 
The concepts and the definitions below can be found in \cite{DFOBook}.

Lagrange polynomials and associated concepts will be defined with respect to a space of polynomials ${\cal P}$ of dimension $p$.
Typically ${\cal P}$ is either the set of linear or quadratic polynomials,  but it also can be a set of quadratic polynomials with a pre-defined Hessian sparsity pattern.

\begin{definition}{\bf Lagrange Polynomials }

Given a space of polynomials ${\cal P}$ of dimension $p$ and a set of points ${\cal Y}=\{y_1,\ldots,y_{p}\}\subset \R^n$,
a set of $p$ polynomials $\ell_j(s)$ in ${\cal P}$ for $j=1,\ldots,p$,
is called a basis of Lagrange polynomials associated with ${\cal Y}$, if
\[
\ell_j(y_i) = \delta_{ij} = \left \{\begin{array}{c} 1
\;\;\; \mbox{if} \;\;\; i=j,\\
0 \;\;\; \mbox{if} \;\;\; i\neq j.
\end{array} \right .
\]
If the basis of Lagrange polynomials exists for the given ${\cal Y}$ 
then ${\cal Y}$ is said to be \emph{poised}. 
\end{definition}

\begin{definition}{\bf $\Lambda$--poisedness  }
Given a space of polynomials ${\cal P}$ of dimension $p$, $\Lambda > 0$, and a set ${\cal B} \subset \R^n$.
A poised set ${\cal Y} = \{ y_1,\ldots,y_{p} \}$ is said to
be $\Lambda$--poised in ${\cal B}$ if ${\cal Y}\subset {\cal B}$ and for the basis of Lagrange polynomials associated with ${\cal Y}$, it holds that
\[
\Lambda \; \geq \; \max_{j=1,\ldots, p} \max_{s\in {\cal B}} |\ell_j(s)|.
\]
\end{definition}

We now show how $\Lambda$--poisedness of the interpolation set can ensure that related interpolation model is fully linear 
and derive corresponding constants $\kappa_{ef}$ and $\kappa_{eg}$. 
%First we note the following simple result (see e.g. \cite{ChaudhryScheinberg2026ICM})
%\begin{lemma} \label{lem:kappa_ef}
% If for model $m_k$ defined by \eqref{eq:model_def}
%      \begin{equation}\label{eq:kappa_eg}
%\|g_k-\nabla \phi(x_k)\|\leq \kappa_{eg}\Delta_k
%\end{equation}
%and $ \|H_k\| \le \kappa_{bhm}$ (from Assumption \ref{assum:tr}) then  $m_k$ is $\kappa_{ef},\kappa_{eg}$-fully linear in $B(x_k,\Delta_k)$ with $\kappa_{ef}=\kappa_{eg}+\frac{L+\kappa_{hbm}}{2}$
%\end{lemma}  
%
Throughout this section we apply Assumptions \ref{assum:lip_cont} and \ref{assum:tr}. By Lemma \ref{lem:fully-lin}, all we need is to ensure \eqref{eq:kappa_eg}. 
For this we specify a way to construct the model $m(x)$. 

%We impose the interpolation conditions by constructing a polynomial $r(x)\in {\cal P}$ such that
%\begin{equation}\label{eq:interp_cond}
%r_k(y)=f(x_k+y)-f(x_k),\quad \forall y\in {\cal Y}_k. 
%\end{equation}
%
%Note that when $|{\cal Y}_k|<p$ $r_k(x)$ is not uniquely defined. There are many practically interesting alternatives of how $r(x)$ should be selected. For example, in the case when ${\cal P}$ is the space of quadratic polynomials and $|{\cal Y}_k|<\frac{n(n+1)}{2}+n$ the choices include selecting a quadratic model with the smallest Frobenius norm of the Hessian  \cite{DFOBook}, smallest Frobenius norm of the {\em change} of the Hessian \cite{MJDPowell_2004}, sparse Hessian \cite{DFOTRpaper}, etc. For the purposes of the theory we explore here, these choices matter only if it can be established that they guarantee fully-linear models. We will not explore these specific choices here and will rely only on fully-linear models with  $|{\cal Y}_k|=p$. We refer the reader to \cite{DFOTRpaper} for further details on $|{\cal Y}_k|<p$ case. 
%
%We will mainly focus on the model that is constructed via linear interpolation, where $r_k(s)=g_k^Ts$ for some $g_k$
%defined by \eqref{eq:interp_cond}. 

Let ${\cal Y}_k$ be a set of  points and ${\cal P}$ be a  space of  polynomials.  
Assuming that  ${\cal Y}_k = \{  y_1,\ldots,y_{p} \}$ 
	is poised  in $B(0, \Delta)$ we define 
	 $g_k$ and $H_k$ to satisfy $\|H\| \leq \kappa_{bhm}$ and 
	\begin{equation}\label{eq:interpolation_acc}
	g_k^\top y + \frac{1}{2}y^\top H_ky = \phi(x_k + y) - \phi(x_k), \quad  \forall y \in {\cal Y}_k.
\end{equation}
The  model $m_k(x)$ is then defined as 
\begin{equation}\label{eq:model_def2}
m_k(x_k+s)=\phi(x_k)+g_k^Ts+\frac{1}{2}s^TH_ks. 
\end{equation}

We now show an important  result that establishes a bound on $\kappa_{eg}$ when $p=n$ and ${\cal Y}_k = \{  y_1,\ldots,y_{n} \}$ is $\Lambda$-poised for homogeneous linear interpolation. This result is an extension of a similar result in \cite{ChaudhryScheinberg2026ICM} which allows $m_k$ to include a  quadratic term.

\begin{theorem}\label{thm:Lambda_to_kappaeg}
	Let ${\cal Y} = \{  y_1,\ldots,y_{n} \}$ be such that ${\cal Y}$
	is $\Lambda$--poised  in $B(0, \Delta)$.
	Let $g,H$ satisfy $\|H\| \leq \kappa_{bhm}$ and $g^\top y + \frac{1}{2}y^\top Hy = \phi(x + y) - \phi(x),$ for each $y \in {\cal Y}$.
	Then
	\[
	\|\nabla \phi(x) - g\| \leq  \frac{1}{2}\left(L + \kappa_{bhm}\right) \Delta \sqrt{n} \sqrt{n (\Lambda^2 - 1) + 2}.
	\]
	In particular, if $\Lambda = 1+O(1/n)$, then we will have $\|\nabla \phi(x) - g\|=O(\sqrt{n})\Delta$.
\end{theorem}
\begin{proof}
	Let $\bar \phi(s)=\phi(x+s)-\phi(x)$, then $\bar \phi(0)=0$ and $\nabla \phi(x)=\nabla \bar \phi(0)$.
	Let $Y$ be the matrix with $i$th column equal to $y_i$, let $D$ be the diagonal matrix such that $D_{ii} = \|y_i\|$, and let $\bar \phi(Y)$ be the vector with $i$th entry equal to $\bar \phi(y_i)$.
	The interpolation condition is
	\[
	g^Ty_i = \phi(x+y_i)-\phi(x) - \frac{1}{2}y_i^T H y_i, \quad i=1, \ldots, n
	\]
	thus we have $g = Y^{-T} \bar \phi(Y) - Y^{-T}h(Y)$, where $h(Y)$ is a vector with components $\frac{1}{2}y_i^T H y_i$.
	Then we have $\|\nabla \phi(x) - g\| \leq \|\nabla \phi(x) - Y^{-T} \bar \phi(Y)\| + \|Y^{-T}h(Y) \|$.
	By the proof of Theorem~4.3 in \cite{ChaudhryScheinberg2026ICM}, we have $\|Y^{-T}D\| \leq \sqrt{n (\Lambda^2 - 1) + 2}$ and
	\[
	\|\nabla \phi(x) - Y^{-T} \bar \phi(Y)\| \leq \frac{1}{2}\sqrt{n} L \Delta \sqrt{n (\Lambda^2 - 1) + 2}
	\]
	and we need only bound $\|Y^{-T}h(Y) \|$.
	We have $\|Y^{-T}h(Y) \| \leq \| Y^{-T}D \| \|D^{-1}h(Y)\|$.
	Observe that the $i$th element of $D^{-1}h(Y)$ is $\frac{1}{2}y_i^T H y_i / \|y_i\| $ and is therefore bounded by $\frac{1}{2}\kappa_{bhm}\Delta$.
	Now we have $\|D^{-1}h(Y)\| \leq \sqrt{n} \|D^{-1}h(Y)\|_\infty \leq \frac{1}{2}\sqrt{n} \kappa_{bhm} \Delta$.
	Thus $\|Y^{-T}h(Y) \| \leq \frac{1}{2}\sqrt{n} \kappa_{bhm} \Delta \sqrt{n (\Lambda^2 - 1) + 2}$, and in total
	\[
	\|\nabla \phi(x) - g\| \leq \frac{1}{2}\left(L + \kappa_{bhm}\right) \Delta \sqrt{n} \sqrt{n (\Lambda^2 - 1) + 2}.
	\]
	
\end{proof}

This theorem allows us to conclude that our model is fully linear when our interpolation set is $\Lambda$-poised.
\begin{corollary}
If $m_k$ is defined as in \eqref{eq:model_def2} and if ${\cal Y}_k$ is $\Lambda$-poised, then $m_k$ is $\kappa_{ef}, \kappa_{eg}$-fully-linear with
\begin{align*}
\kappa_{eg} &= \frac{1}{2}\left(L + \kappa_{bhm}\right) \sqrt{n} \sqrt{n (\Lambda^2 - 1) + 2} \\
\kappa_{ef} &= \kappa_{eg}+\frac{L+\kappa_{bhm}}{2}.
\end{align*}	
\end{corollary}
\begin{proof}
This follows from the definition of $m_k$ in \eqref{eq:model_def2}, Theorem~\ref{thm:Lambda_to_kappaeg} and Lemma~\ref{lem:fully-lin}.
\end{proof}

We note that constructing the model to satisfy \eqref{eq:interpolation_acc} is not possible unless $\phi(x + y) - \phi(x)$ can be computed exactly. In the case of inexact oracles we instead compute $g$ and $H$ from
	\begin{equation}\label{eq:interpolation}
	g_k^\top y + \frac{1}{2}y^\top H_ky = f(x_k + y) - f(x_l) \quad \forall y \in {\cal Y}_k.
\end{equation}
The previous theorem is modified as follows. 
\begin{theorem}\label{thm:Lambda_to_kappaeg_err}
	Let ${\cal Y} = \{  y_1,\ldots,y_{n} \}$ be such that ${\cal Y}$
	is $\Lambda$--poised  in $B(0, \Delta)$.
	Let $g,H$ be computed to satisfy \eqref{eq:interpolation} and $\|H\| \leq \kappa_{bhm}$. 
	Then
	\[
	\|\nabla \phi(x) - g\| \leq  \sqrt{n (\Lambda^2 - 1) + 2}\left(\frac{1}{2}\left(L + \kappa_{bhm}\right)\sqrt{n}  \Delta +\sqrt{n} \frac{2 \epsilon_f \Lambda }{\Delta}\right).
	\]
\end{theorem}
\begin{proof}
	Define $\bar \phi$, $D$, $\bar \phi(Y)$, $h(Y)$ as in the previous proof.
	Diverging from that proof, the interpolation condition changes to  
	\[
	g^Ty_i = \phi(x+y_i)-\phi(x) - \frac{1}{2}y_i^T H y_i +(f(x+y_i)-\phi(x+y_i))-(f(x)-\phi(x)), \quad i=1, \ldots, n
	\]
	thus we have $g = Y^{-T} \bar \phi(Y)  - Y^{-T}h(Y) + Y^{-T} E$, where $E$ is a vector with components $(f(x+y_i)-\phi(x+y_i))-(f(x)-\phi(x))$.
	Then we can bound the error
	\begin{align*}
		\|\nabla \phi(x) - g\| &= \|Y^{-T} D (D^{-1} Y^T \nabla \bar \phi(0) - D^{-1} \bar \phi(Y) - D^{-1} E)\|  + \|Y^{-T}h(Y) \|\\
		&\leq \sqrt{n} \|Y^{-T} D\| \| D^{-1} Y^T \nabla \bar \phi(0) - D^{-1} \bar \phi(Y) - D^{-1} E\|_{\infty}  + \|Y^{-T}h(Y) \|\\
		&\leq \sqrt{n} \|Y^{-T} D\| \left( \| D^{-1} Y^T \nabla \bar \phi(0) - D^{-1} \bar \phi(Y)\|_{\infty} + \| D^{-1} E\|_{\infty} \right)  + \|Y^{-T}h(Y) \|\\
		&\leq \sqrt{n} \|Y^{-T} D\| \left( \| D^{-1} Y^T \nabla \bar \phi(0) - D^{-1} \bar \phi(Y)\|_{\infty} + \| D^{-1}\|_{\infty} \| E\|_{\infty} \right)  + \|Y^{-T}h(Y) \|.
	\end{align*}
	We can bound $\|Y^{-T} D\|$,$\| D^{-1} Y^T \nabla \bar \phi(0) - D^{-1} \bar \phi(Y)\|_{\infty}$, $\|Y^{-T}h(Y) \|$ identically as in the previous proof.
	To bound $\| E\|_{\infty}$ observe that the condition that $|f(x+y)-\phi(x+y)| \leq\epsilon_f$ for $y \in {\cal Y} \cup \{0\}$ implies $\| E\|_{\infty} \leq 2 \epsilon_f$.
	To bound $\| D^{-1}\|_{\infty}$ recall that $D_{ii} = \|y_i\|$.
	By the properties of Lagrange polynomials we have $\ell_i(y_i) = 1$.
	By linearity, we have $\ell_i(\Delta \frac{y_i}{\|y_i\|}) = \frac{\Delta}{\|y_i\|}$.
	By the poisedness condition we have $\left |\ell_i(\Delta \frac{y_i}{\|y_i\|}) \right |\leq \Lambda$.
	Combining these bounds we have $\frac{\Delta}{\|y_i\|} \leq \Lambda$ and thus $\frac{1}{\|y_i\|} \leq \frac{\Lambda}{\Delta}$.
	Thus $\| D^{-1}\|_{\infty} \leq  \frac{\Lambda}{\Delta}$.
	The result follows.
\end{proof}

This theorem allows us to conclude that our model is fully linear when our interpolation set is $\Lambda$-poised and when $\Delta$ is sufficiently large.
\begin{corollary}\label{thm:fully_linear_lambda}
	If $m_k$ is defined as in \eqref{eq:model_def2}, if ${\cal Y}_k$ is $\Lambda$-poised, and if $\Delta_k \geq \sqrt{\frac{4\epsilon_f \Lambda}{L + \kappa_{bhm}}}$, then $m_k$ is $\kappa_{ef}, \kappa_{eg}$-fully-linear with
	\begin{align*}
		\kappa_{eg} &= \left(L + \kappa_{bhm}\right) \sqrt{n} \sqrt{n (\Lambda^2 - 1) + 2} \\
		\kappa_{ef} &= \kappa_{eg}+\frac{L+\kappa_{bhm}}{2}.
	\end{align*}	
\end{corollary}
\begin{proof}
This follows from the definition of $m_k$ in \eqref{eq:model_def2}, Theorem~\ref{thm:Lambda_to_kappaeg_err} and Lemma~\ref{lem:fully-lin}, since the bound on $\Delta_k$ implies that 
\[
\sqrt{n (\Lambda^2 - 1) + 2}\left(\frac{1}{2}\left(L + \kappa_{bhm}\right)\sqrt{n}  \Delta_k +\sqrt{n} \frac{2 \epsilon_f \Lambda }{\Delta_k}\right) \leq 2\frac{1}{2}\left(L + \kappa_{bhm}\right) \Delta_k \sqrt{n} \sqrt{n (\Lambda^2 - 1) + 2}.
\]
\end{proof}

In the next section we propose a concrete algorithm for constructing  $\Lambda$-poised sample sets and deriving the bound on the number of iterations required to do so.

\section{Model improving iterations based on Lagrange polynomials.}
\label{sec:powell}

In \cite{ChaudhryScheinberg2026ICM} a specific algorithm based on the framework of Algorithm \ref{alg:tr} was proposed and its complexity analyzed. This algorithm maintains a set ${\cal Y}$ of $n$ sample points and a set of linear Lagrange polynomials. 
At each model improving iteration the algorithm performs, what we call, a {\em geometry correcting step} by either replacing an interpolation point outside the trust region, if such a point exists, or  replacing a point whose corresponding Lagrange polynomial violates the $\Lambda$-poisedness condition. If no such improvement is possible, then the set is $\Lambda$-poised and thus the model is fully-linear, by Theorems \ref{thm:Lambda_to_kappaeg} and \ref{thm:Lambda_to_kappaeg_err}, hence the iteration is unsuccessful and the trust region radius is reduced. It is shown in  \cite{ChaudhryScheinberg2026ICM} that the number of geometry correcting iterations between any two other iterations is at most $3n$. Each such iteration performs at most two function evaluations, while successful and unsuccessful iterations perform at most one.  Thus using Theorem \ref{thm:tr_complexity_ef} one can derive the bound on $|{\cal M}_\epsilon|$ and consequently the bound on the total complexity of the algorithm.

The method in \cite{ChaudhryScheinberg2026ICM}  fails to include several important practical elements. Firstly, the models that are being constructed by any successful model-based DFO method are quadratic and are usually based on quadratic interpolation. Secondly, since the function value is computed at the trial step, even if the step is not accepted as the new iterate, the step provides a new sample point which may improve the interpolation model. In fact it is likely to do so, since the reason the step is rejected, to begin with, is due to disagreement of the model and the function at the trial step. Thus adding this trial step to the new model  provides new information. The improvement guaranteed by replacing some interpolation point by the trial step  can be quantified by the value of the corresponding Lagrange polynomial at the trial step.  A principled algorithm relying on this property, which is referred to as {\em self-correcting}, has been developed  in \cite{scg}  and shown to converge to a stationary point in the limit. This algorithm uses only such {\em self-correcting steps} for model improvement  and does not perform geometry correction steps. It also includes the criticality step which is not practical but helps the analysis.   No complexity bounds have been developed in \cite{scg}.  Powell utilized  both the self-correcting and geometry correcting steps in his algorithms \cite{MJDPowell_2004b,powell2001lagrange}. In what follows we present an algorithm that allows models to be constructed using quadratic interpolation and utilizes the self-correcting and geometry correcting steps. On the other hand it only maintains a set of linear Lagrange polynomials and ensures only a subset of interpolation points to be $\Lambda$-poised for linear interpolation. Combined with the results of the previous section this allows the algorithm to ensure eventual construction of fully-linear models. After we present the algorithm we 
state and prove the bound on the  number of consecutive model improving iterations. 

Algorithm \ref{alg:powell} is the description of our proposed method that on the one hand tries to include most of the practical  elements of Powell's methods and on the other hand nearly matches the complexity of the simplified method in \cite{ChaudhryScheinberg2026ICM}. 

At each iteration the method maintains two sets of points - set ${\cal Y}$ of $n$ points whose geometry is monitored and maintained by  means of the associated linear Lagrange polynomials   and another set ${\cal Z}$ of 
$p\leq (n-1)n/2$ points whose geometry is only monitored and maintained in terms of distance to the trust region center.

We construct the model by solving the following constrained least squares problem. 

\begin{align}\label{eqn:model-regression}
     \min_{g, H} & \sum_{z \in {\cal Z}_k} \left(f(x_k) + g^\top z + \frac{1}{2} z^T H z - f(x_k + z) \right)^2 \\
     \text{subject to: } &\|H\| \leq K\\
     & g^\top y + \frac{1}{2}y^\top Hy = f(x_k + y) - f(x_k), \quad  \forall y \in {\cal Y}_k.
\end{align}

Any solution to this problem will satisfy $\|H\| \leq \kappa_{bhm}$ with $\kappa_{bhm}\leq K$. If a quadratic  model exists whose Hessian satisfies
$\|H\| \leq \kappa_{bhm}$ with $\kappa_{bhm}\leq K$ and that interpolates  all points in ${\cal Y}$ and ${\cal Z}$, then such model will be an optimal solution to this problem.\footnote{We deliberately distinguish $K$ and $\kappa_{bhm}$, since $K$ is chosen by the algorithm as will be set to be large, while $\kappa_{bhm}$ is the true bound that occurs during the algorithm and will depend on the problem.} When $p<\frac{n(n-1)}{2}$ then the problem may have multiple optimal solutions. In the case when this happens we can select the solution with the smallest Frobenius norm (see \cite{DFOBook}), as long as it satisfies $\|H\| \leq K$. Alternatively, following Powell's ideas from \cite{MJDPowell_2004b} we can select the solution for which $H$ is the closest in Frobenius norm to the Hessian from the previous iteration.  We will discuss both of these techniques in the computational section. 

%We also consider the following iterative method:
%
%\begin{algorithm}[h]
%     \caption{Iterative Hessian Truncation}
%     {\bf Inputs:} Number of iterations $J \in \{1,2,\dots\}$ \\
%     \nl Solve for $g$ and subsequently $H$ via the interpolation conditions \begin{align*}
%          g^\top y = f(x_k + y) - f(x_k) \quad \forall y \in {\cal Y}_k,  \qquad \frac{1}{2} z^\top H z = f(x_k + z) - f(x_k) - g^\top z \quad \forall z \in {\cal Z}_k
%     \end{align*}
%     \nl 
%          \For{$i=1,\dots, J$}{
%               \If{$\|H\| \geq \kappa_{bhm}$}{
%                    $\bar H$ = Truncate($H$)
%               }
%               \Else{$\bar H = H$}
%               Solve for $\bar g$ subject to the interpolation conditions $\bar g^Ty + \frac{1}{2} y^\top \bar H y = f(x_k + y) - f(x_k), \forall y \in {\cal Y}_k$\\
%               Solve for $\tilde H$ subject to the interpolation conditions $\frac{1}{2} z^\top \tilde H z = f(x_k + z) - f(x_k) - \bar g^\top z,  \forall z \in {\cal Z}_k$ \\
%               Set $g = \bar g$ and $H = \tilde H$ 
%          }
%\end{algorithm}
%
%The truncation can be performed by taking an SVD and clipping the eigenvalues  that exceed $\kappa_{bhm}$. 
%

\begin{algorithm}[H]
\caption{~\textbf{Geometry-correcting algorithm}}
\label{alg:powell}
{\bf Inputs:} A zeroth-order oracle $f(x)\approx {\phi}(x)$,   $p=\frac{(n-1)n}{2}$, $\Delta_0$,  $x_0$,  $\gamma \in (0,1)$ $\eta_1 >0$, $\eta_2 >0$, $\Lambda >1$, $\Lambda_{sc} \geq 1$.\\ 
{\bf Initialization} An initial set ${\cal Y}_0$ such that $|{\cal Y}_0| =n$, an initial set ${\cal Z}_0$ such that $|{\cal Z}_0| \leq p$ and the function values $f(x_0)$, $f(x_0 + y_i)$, $y_i \in {\cal Y}_0$, $f(x_0 + z_i)$, $z_i \in {\cal Z}_0$. A set of Lagrange Polynomials  
$\{\ell_i(x), i=1, \ldots, n\}$ in  ${\cal P}$ for the set ${\cal Y}_0$.\\
\For{$k=0,1,2,\dots$} {
	\nl Build a quadratic model $m_k(x_k + s)$ as in \eqref{eq:model_def2} using $f(x_k)$ and 
	$f(x_k + y_i)$, $y_i\in {\cal Y}_k$, $f(x_0 + z_i)$, $z_i \in {\cal Z}_k$.
	 \\
	\nl Compute a trial step $s_k$ and  ratio $\rho_k$  as in Algorithm \ref{alg:tr}. \\
       \nl  {\em Successful iteration:} $\rho_k\geq \eta_1$ and     $\|g_k\|\geq \eta_2\Delta_k$. \\
      Set $x_{k+1} = x_k+s_k$, $\Delta_{k+1}=\gamma^{-1}\Delta_k$. \\
 \[
{j_k^* =\arg \max_{j=1,\ldots, n} \|y_j-s_k\|},\quad  {i_k^* =\arg \max_{i=1,\ldots, p} \|z_i-s_k\|.}
\]
If $\|y_{j_k^*}-s_k\|> \|z_{i_k^*}-s_k\|$\footnote{If $ {\cal Z}_k\neq \emptyset$ check $\|y_{j_k^*}-s_k\|>0$} and $|\ell_{j_k^*}(s_k)|>0$ then update set ${\cal Y}_{k+1} = ({\cal Y}_k\setminus \{y_{j_k^*}\} \cup \{0\}) - s_k$. Recompute Lagrange Polynomials for ${\cal Y}_{k+1}$.  Otherwise update set ${\cal Z}_{k+1} = ({\cal Z}_k\setminus \{z_{i_k^*}\} \cup \{0\}) - s_k$ if $|{\cal Z}_k|=p$,   \\
or ${\cal Z}_{k+1} ={\cal Z}_k \cup \{0\}-s_k$ if $|{\cal Z}_k|<p$.\\
\nl {\em Model improving or unsuccessful iteration:} $\rho_k < \eta_1$ or  $\|g_k\|<\eta_2\Delta_k$. Set $x_{k+1}=x_k$, $I_{imp}=0$ and perform all the applicable steps below. 
\begin{description}
\item[(i)] {\em Geometry correction by  replacing a far point in ${\cal Y}$:}
Let $j_k^* =\arg \max_{j=1,\ldots, n} \|y_j\|.$\\
 If  $\|y_{j_k^*}\|>\Delta_k$ $\Rightarrow$  $s^*=y_{j_k^*}$ $I_{imp}=1$. If $|\ell_{j_k^*}(s_k)|>0$, then $s_k^*=s_k$, otherwise $s_k^* =\arg \max_{ s\in B(0,\Delta_k)} |\ell_{j_k^*}(s)|$, compute  $f(x_k + s_k^*)$. Update  {${\cal Y}_{k+1} = {\cal Y}_k\setminus \{y_{j_k^*}\} \cup \{s_k^*\}$ and Lagrange polynomials.} 
\item[(ii)] {\em Self-correction by replacing  a point in ${\cal Y}$ with a large Lagrange Polynomial value at the trial step:} 
If $I_{imp}=0$, 
\[
j_k^* =\arg \max_{j=1,\ldots, n} |\ell_j(s_k)|.
\]
 If $|\ell_{j^*_k}(s_k)|>\Lambda_{sc}$, then $s^*=y_{j_k^*}$, update {${\cal Y}_{k+\frac{1}{2}}  = {\cal Y}_k\setminus \{y_{j_k^*}\} \cup \{s_k\}$. Update the set of Lagrange Polynomials for ${\cal Y}_{k+\frac{1}{2}}$.} 
 Otherwise  $s^*=s_k$, ${\cal Y}_{k+\frac{1}{2}}={\cal Y}_k$. 
 \item[(iii)] {\em Geometry correction of ${\cal Y}$ by replacing a "bad" point:} If $I_{imp}=0$ 
 %construct the set Lagrange Polynomials  
%$\{\ell_j(x), i=1, \ldots, n\}$ in  ${\cal P}$ for the set ${\cal Y}_k$ and  find max  value 
compute
\[
{(j_k^*,s_k^*) =\arg \max_{j=1,\ldots, p, s\in B(0,\Delta_k)} |\ell_j(s)|.}
\]
 If $|\ell_{j_k^*}(s_k^*)|>\Lambda$, compute  $f(x_k + s_k^*)$, ${\cal Y}_{k+1} = {\cal Y}_{k+\frac{1}{2}}\setminus \{y_{j_k^*}\} \cup \{s_k^*\}$. Update the set of Lagrange Polynomials for ${\cal Y}_{k+1}$ and set $I_{imp}=1$. Otherwise  ${\cal Y}_{k+1} = {\cal Y}_{k+\frac{1}{2}}$. \\

\item[(iv)]  {\em Attempt to improve $Z$ using available points.} For each defined $s\in \{s^*, s_k^*\}$ repeat:
\begin{itemize}
 \item {\em Improvement to ${\cal Z}$ by  adding a point:} If $|{\cal Z}_k|<p$,  ${\cal Z}_{k+1} ={\cal Z}_k \cup \{s\}$.
\item {\em Improvement  to ${\cal Z}$  to by replacing a far point:} Else, let $i_k^* =\arg \max_{i=1,\ldots, n} \|z_i\|.$\\
 If  $\|z_{i_k^*}\|>\|s\|$ $\Rightarrow$   {${\cal Z}_{k+1} = {\cal Z}_k\setminus \{z_{i_k^*}\} \cup \{s\}$. }
 \item {\em Keep  ${\cal Z}$: } Otherwise $Z_{k+1}= {\cal Z}_{k} $. 
 \end{itemize}
\item {\em Model improving iteration:} If $I_{imp}=1$,  $\Delta_{k+1} = \Delta_k$. 
\item {\em Unsuccessful iteration:} If $I_{imp}=0$,  $\Delta_{k+1} =\gamma \Delta_k$. 
\end{description}
%\nl 
%\nl {Unsuccessful iteration, $\|g_k\|/\Delta_k$} is small ] \\
% Set $x_{k+1}=x_k$, $\Delta_{k+1} =\gamma \Delta_k$.\\
%${\cal Y}_{k+1} = {\cal Y}_k\setminus \{y_{k,\max}\} \cup \{x_k+s_k\}$  
%
}
\end{algorithm}

We note that case (iii) of Step 4 of Algorithm~\ref{alg:powell} is equivalent to one iteration of the geometry correction procedure described in \cite{DFOBook}.

In  cases (ii) and (iii) of Step 4 point $y_{j_k^*}$ in the current ${\cal Y}$ set is replaced by a new point, let's call it $\tilde s$.    The update for deriving the new 
Lagrange Polynomial set  $\ell^+$
can be carried out via the following formulae: 
\begin{align}\label{eq:LPupdate}
\ell^+_{j^*_k}(x) &= \frac{\ell_{j^*_k}(x)}{\ell_{j^*_k}(\tilde s)}\\
\ell^+_{i}(x) &= \ell_{i}(x) - \ell_{i}(\tilde s)\ell^+_{j^*_k}(x) \quad i \neq j^*_k. 
\end{align}

Note that since we  assume that ${\cal Y}_0$ is poised then so are all consequent  ${\cal Y}_k$ sets by construction.

  We now provide results that allow us to bound the number of model improving iterations  $ {\cal M}_\epsilon$.

\begin{theorem}\label{thm:geom_correct_total_lin}
Let ${\cal P}$ be the set of linear polynomials (with dimension $p=n$).
Then the number of oracle calls in any sequence of consecutive model improving iterations; i.e. such that $k\in  {\cal M}_\epsilon$ is at most $4 n\log n + 8n + 2n|\log\log(\Lambda)|$.
\end{theorem}
We break the proof of this theorem into the following two lemmas, the first of which derives a bound on the number of consecutive 
model improving iterations required to obtain a desired $\Lambda$-poised set starting from a $\Lambda_0$-poised set for an arbitrary $\Lambda_0$. 
The second lemma shows that after the first $2n$ consecutive model improving iterations $\Lambda_0$-poised set is obtained with a particular bound on $\Lambda_0$. 

\begin{lemma}\label{lem:geom_correct_2}
Let ${\cal P}$ be the set of linear polynomials (with dimension $p=n$).
If ${\cal Y}_k$ is $\Lambda_0$-poised in $B(0,\Delta)$, then there will be at most $\left \lceil
n\log n + n|\log \log \Lambda_0| + n|\log\log(\Lambda)|
\right \rceil$ additional consecutive model improving iterations.
\end{lemma}
\begin{proof}
Assume for simplicity and w.l.o.g that $\Delta=1$. 
Having a $\Lambda_0$-poised set implies by Hadamard's inequality that we have 
\[
|\det(Y^{-T})| \leq \prod_{i}^{n} \|(Y^{-T})_i\| \leq \Lambda_0^n.
\]
since $\max_{x \in B(0,1)} |\ell_i(x)| =  \|(Y^{-T})_i\|$.
Thus $|\det(Y)| \geq \Lambda_0^{-n}$.
We have by Cramer's rule that replacing $y_i$ with $s$ results in a matrix $Y^+$ which satisfies
\[
|\det(Y^+)| = |\det(Y)|  |\ell_i(s)|.
\]
Combining this with the fact that $|\det(Y)| \leq 1$, we have that $\max_{i \in [n]}\max_{x \in B(0,1)} |\ell_i(x)| \leq |\det(Y)|^{-1}$.
Thus $\mathcal{Y}$ is $\Lambda$-poised if $|\det(Y)| \geq \Lambda^{-1}$.
All that remains is to show that $|\det(Y)|$ must increase quickly.

After ``Geometry correction of ${\cal Y}$ by replacing a "bad" point'' resulting in a matrix $Y^+$, we have
\[
|\det(Y^+)| = |\det(Y)| \left( \max_{i \in [n]}\max_{x \in B(0,1)} |\ell_i(x)| \right) = |\det(Y)| \max_{i \in [n]} \|(Y^{-T})_i\|.
\]
Observe by the AM-GM inequality
\[
\max_{i \in [n]} \|(Y^{-T})_i\|^2 \geq \frac{1}{n} \sum_i \|(Y^{-T})_i\|^2 = \frac{1}{n} \|Y^{-T}\|_F^2 = \frac{1}{n} \sum_i \sigma_i(Y^{-T})^2 \geq \sqrt[n]{\sigma_i(Y^{-T})^2 } = |\det(Y)|^{-\frac{2}{n}}.
\]
Thus we have $\max_{i \in [n]}\max_{x \in B(0,1)} |\ell_i(x)| = \max_{i \in [n]} \|(Y^{-T})_i\| \geq |\det(Y)|^{-\frac{1}{n}}$ and hence
\[
|\det(Y^+)| \geq |\det(Y)| |\det(Y)|^{-\frac{1}{n}} = |\det(Y)|^{\left( 1-\frac{1}{n}\right)}.
\]
Taking the logarithm of both sides, we have
\[
\log|\det(Y^+)| \geq \left( 1-\frac{1}{n}\right) \log|\det(Y)|.
\]
Thus, $-\log |\det(Y)|$ shrinks exponentially.

If in a single iteration we first do ``Geometry correction by replacing a point in ${\cal Y}$ with a Lagrange Polynomial value'' to get $Y^+$, and then do ``Geometry correction of ${\cal Y}$ by replacing a "bad" point'' to get $Y^{++}$ we have
\[
\log|\det(Y^{++})| \geq \left( 1-\frac{1}{n}\right) \log|\det(Y^+)| \geq \left( 1-\frac{1}{n}\right) \log \left(|\det(Y)| |\ell_i(s^*_k)| \right)\geq \left( 1-\frac{1}{n}\right) \log|\det(Y)|.
\]
Thus, in either case, $-\log |\det(Y)|$ shrinks exponentially.

Recall that we have $|\det(Y)| \geq \Lambda_0^{-n}$ and thus $\log |\det (Y)| \geq -n \log(\Lambda_0)$.
Also, we will have achieved $\Lambda$-poisedness if $|\det(Y)| \geq \Lambda^{-1}$ or equivalently $\log|\det(Y)| \geq -\log(\Lambda)$.
Thus, the number of additional iterations is bounded by
\[
\left \lceil
\frac{
\log \left( \frac{-n \log(\Lambda_0)}{-\log(\Lambda)} \right)
}{
\log \left( 1-\frac{1}{n}\right)
}
\right \rceil
\leq
\left \lceil
n \log \left( \frac{n \log(\Lambda_0)}{\log(\Lambda)} \right)
\right \rceil
\leq
\left \lceil
n\log n + n |\log \log \Lambda_0| + n|\log\log(\Lambda)|
\right \rceil.
\]
\end{proof}

\begin{lemma}\label{lem:geom_correct_1}
Let ${\cal P}$ be the set of linear polynomials (with dimension $p=n$).
Then ${\cal Y}_k$ is $14^n$-poised after at most $2n$ consecutive model improving iterations.
\end{lemma}
\begin{proof}
Because the first step of a model improving iteration replaces any point in ${\cal Y}_k$ that is outside $B(0,\Delta_k)$ then after at most $n$ such 
iterations (and $n$ oracle calls),  ${\cal Y}_k$ contains $n$ points all of which have norm at most $\Delta_k$.
For simplicity, we assume $\Delta_k=1$.
% Since the maximum of a linear function over the unit ball is always attained on the boundary, then new points added to  ${\cal Y}_k$ are at the boundary, thus  ${\cal Y}_k$ contains $n$ unit vectors.
%
We show that in at most $n$ additional iterations, the set $\mathcal{Y}_k$ becomes $14^n$-poised.
%Second, we show that after achieving $14^n$-poisedness, at most $O(n\log n + n |\log \log \Lambda|)$ additional iterations are required to achieve $\Lambda$-poisedness.

For a subspace $S$, we say that a set of points ${\cal Y}$ is \emph{$\Lambda$-poised in $S$} if for the Lagrange polynomials $\ell_i$, we have $\max_{x \in B(0,1) \cap S} |\ell_i(x)| \leq \Lambda$.
Note that if $y_1 \in {\cal Y}_k$ is a unit vector, then ${\cal Y}$ is $1$-poised in $S = \text{span}(\{y_1\})$.
Below, we prove two claims regarding poisedness in a subspace.

First, we claim that if ${\cal Y}$ is $\Lambda$-poised in $S$, then after step (ii): ``Self-correction by replacing a point in ${\cal Y}$ with a large Lagrange Polynomial value'', we will have that the new set of points is $2\Lambda$-poised in $S$.

Let $s_k$ be the point we are adding and let $j_k^* =\arg \max_{j=1,\ldots, n} |\ell_j(s_k)|$.
For the replacement to take place, we must have $|\ell_{j_k^*}(s_k)| \geq 1$.
Let $\ell^+_i$ denote the Lagrange polynomials after replacement.
Then from \eqref{eq:LPupdate} we have
\[
\max_{x \in B(0,1) \cap S} |\ell^+_{j_k^*}(x)| = \max_{x \in B(0,1) \cap S} \frac{|\ell_{j_k^*}(x)|}{|\ell_{j_k^*}(s_k)|} \leq \frac{\Lambda}{1} \leq 2\Lambda.
\]
For $i \neq j_k^*$, we have
\begin{align*}
\max_{x \in B(0,1) \cap S} |\ell^+_i(x)| &= \max_{x \in B(0,1) \cap S} \left| \ell_i(x) - \frac{\ell_{i}(s_k)\ell_{j_k^*}(x)}{\ell_{j_k^*}(s_k)} \right|\\& \leq \max_{x \in B(0,1) \cap S} |\ell_i(x)| + \left| \frac{\ell_{i}(s_k)}{\ell_{j_k^*}(s_k)} \right| \left( \max_{x \in B(0,1) \cap S} |\ell_{j_k^*}(x)|  \right) \leq 2\Lambda.
\end{align*}
The last inequality is because $ |\ell_{i}(s_k)|\leq |\ell_{j_k^*}(s_k)|$. This completes the proof of this first claim.

Second, we claim that if ${\cal Y}$ is $\Lambda$-poised in a proper subspace $S$, then after (iii): ``Geometry correction of ${\cal Y}$ by replacing a "bad" point'', we will have that the new set of points is $7\Lambda$-poised in $S^+$, where $S^+$ is a subspace of one dimension greater than $S$.

Let $(j_k^*,s_k^*) =\arg \max_{j=1,\ldots, p, s\in B(0,\Delta_k)} |\ell_j(s)|$.
We will replace $y_{j_k^*}$ with $s^*_k$; let $\ell^+_i$ denote the Lagrange polynomials after replacement and apply \eqref{eq:LPupdate}. 

If ${\cal Y}$ is $2\Lambda$-poised in the whole space, then we have
\[
\max_{x \in B(0,1) } |\ell^+_{j_k^*}(x)| = \max_{x \in B(0,1) } \frac{|\ell_{j_k^*}(x)|}{|\ell_{j_k^*}(s^*_k)|} = 1 \leq 7\Lambda.
\]
Furthermore, for $i \neq j_k^*$, we have
\begin{align*}
\max_{x \in B(0,1)} |\ell^+_i(x)| & = \max_{x \in B(0,1)} \left| \ell_i(x) - \frac{\ell_{i}(s^*_k)\ell_{j_k^*}(x)}{\ell_{j_k^*}(s^*_k)} \right| \\&\leq \max_{x \in B(0,1)} |\ell_i(x)| + \left| \frac{\ell_{i}(s^*_k)}{\ell_{j_k^*}(s^*_k)} \right| \left( \max_{x \in B(0,1)} |\ell_{j_k^*}(x)|  \right) \leq 4\Lambda \leq 7 \Lambda.
\end{align*}
Thus, we have that the points are $7\Lambda$ poised in the whole space and we can take $S^+$ to be any subspace of dimension one greater that $S$.
Thus, the claim is shown if ${\cal Y}$ is $2\Lambda$-poised in the whole space.
For the remainder of the proof of the claim we may assume ${\cal Y}$ is not $2\Lambda$-poised in the whole space.

Let $P_S$ and $P_S^\perp$ denote the projections onto $S$ and its orthogonal complement, respectively.
By linearity, we have
\[
2 \Lambda < \|\ell(s^*_k)\|_\infty = \|\ell(P_S^\perp s^*_k) + \ell(P_S s^*_k)\|_\infty \leq \|\ell(P_S^\perp s^*_k)\|_\infty + \|\ell(P_S s^*_k)\|_\infty \leq \|\ell(P_S^\perp s^*_k)\|_\infty + \Lambda.
\]
Thus, $\|\ell(P_S^\perp s^*_k)\|_\infty > \Lambda$.
By homogeneity and optimality we have,
\[
\frac{\|\ell(P_S^\perp s^*_k)\|_\infty}{\|P_S^\perp s^*_k\|} = \left\|\ell\left(\frac{P_S^\perp s^*_k}{\|P_S^\perp s^*_k\|}\right)\right\|_\infty \leq \|\ell(s^*_k)\|_\infty \leq \|\ell(P_S^\perp s^*_k)\|_\infty + \Lambda.
\]
Thus we have
\[
\|P_S^\perp s^*_k\| \geq \frac{\|\ell(P_S^\perp s^*_k)\|_\infty}{\|\ell(P_S^\perp s^*_k)\|_\infty + \Lambda} \geq \frac{1}{2}
\]
since $\|\ell(P_S^\perp s^*_k)\|_\infty > \Lambda$.

By an argument identical to that used in the proof of the first claim, we have $\max_{x \in B(0,1) \cap S} \|\ell^+(x)\|_\infty \leq 2\Lambda$.
Next we wish to bound $\left\|\ell^+\left(\frac{P_S^\perp s^*_k}{\|P_S^\perp s^*_k\|}\right)\right\|_\infty$.
Noting that $\ell^+(s^*_k) = e_{j_k^*}$, we have
\[
\|\ell^+(P_S^\perp s^*_k)\|_\infty = \|\ell^+(s^*_k) - \ell^+(P_S s^*_k)\|_\infty \leq \|\ell^+(s^*_k)\|_\infty + \|\ell^+(P_S s^*_k)\|_\infty = 1 + 2\Lambda \leq 3\Lambda.
\]
Then we have
\[
\left\|\ell^+\left(\frac{P_S^\perp s^*_k}{\|P_S^\perp s^*_k\|}\right)\right\|_\infty = \frac{\|\ell^+(P_S^\perp s^*_k)\|_\infty}{\|P_S^\perp s^*_k\|} \leq 6\Lambda.
\]

Let us define $S^+ = S \oplus \text{span}(\{P_S^\perp s^*_k\})$.
We can represent any vector $x$ in  $B(0,1) \cap S^+$, as $\alpha x' + \beta \frac{P_S^\perp s^*_k}{\|P_S^\perp s^*_k\|}$ where $x' \in B(0,1) \cap S$ and $\alpha^2 + \beta^2 \leq 1$
Then we have
\[
\left\| \ell^+ \left( \alpha x' + \beta \frac{P_S^\perp s^*_k}{\|P_S^\perp s^*_k\|}\right) \right \|_\infty \leq 
\alpha \|\ell^+(x')\|_\infty + \beta \left\| \ell^+ \left(\frac{P_S^\perp s^*_k}{\|P_S^\perp s^*_k\|}\right) \right \|_\infty \leq |\alpha| 2 \Lambda + |\beta| 6\Lambda \leq \left \| \begin{bmatrix} 2 \\ 6 \end{bmatrix} \right \| \Lambda \leq 7 \Lambda.
\]
This completes the proof of the second claim.

For $t \leq n$, we will show by induction that after at most $t$ iterations after removing far points, ${\cal Y}_k$ is $14^t$-poised in a subspace of dimension at least $t$.

For the base case, when ``Geometry correction of ${\cal Y}_k$ by replacing a "bad" point'' is performed for the first time, it is done by maximizing a linear Lagrange polynomial over a ball.
Since the maximum of a linear function over the unit ball is always attained on the boundary, the new point added to  ${\cal Y}_k$ is on the boundary, thus  ${\cal Y}_k$ contains at least one unit vector.
Recall that if $y_1 \in {\cal Y}_k$ is a unit vector, then ${\cal Y}$ is $1$-poised in $S = \text{span}(\{y_1\})$.
Thus, ${\cal Y}_k$ is $1$-poised in a one-dimensional subspace, establishing the base case.

Now let us assume that after the $t$th iteration after removing far points, ${\cal Y}_k$ is $14^t$-poised in a subspace, $S$, of dimension at least $t$, and let us show that after an additional iteration, ${\cal Y}_k$ is $14^{t+1}$-poised in a subspace of dimension at least $t+1$.
In the $t$th-iteration, the algorithm either performs  (ii) ``Self-correction by replacing a point in ${\cal Y}_k$ with a large Lagrange Polynomial value'' and (iii) ``Geometry correction of ${\cal Y}_k$ by replacing a "bad" point'', or just the latter.
If the algorithm performs both corrections, then after the first correction we have that the points are $2 \cdot 14^t$-poised in $S$, by the first claim.
Then after the second correction, by the second claim, the points are $7 \cdot 2 \cdot 14^t = 14^{t+1}$-poised in $S^+$ which is of dimension at least one greater than $S$.
Alternatively, if the algorithm only performs the latter correction, then by the second claim, the points are $7 \cdot 14^t \leq 14^{t+1}$-poised in $S^+$ which is of dimension at least one greater than $S$.
Thus, in either case, the inductive argument is complete.
Thus, after at most $n$ iterations after removing far points, ${\cal Y}_k$ is $14^n$-poised in a subspace of dimension at least $n$ (i.e. the whole space).
Thus, after at most $2n$ total iterations (including removal of far away points), ${\cal Y}_k$ is $14^n$-poised.
\end{proof}

\begin{proof}[Proof of Theorem~\ref{thm:geom_correct_total_lin}]
We can combine Lemmas~\ref{lem:geom_correct_1} and \ref{lem:geom_correct_2} by letting $\Lambda_0 = 14^n$.
Thus the total number of iterations can be bounded as follows:
\begin{align*}
&2n + \left \lceil
n\log n + n |\log \log 14^n| + n|\log\log(\Lambda)|
\right \rceil \\
&= 2n + \left \lceil
2 n\log n + n |\log \log 14| + n|\log\log(\Lambda)|
\right \rceil \\
&\leq 2 n\log n + 4n + n|\log\log(\Lambda)|.
\end{align*}
Since each iteration uses at most $2$ oracle calls, the result follows.
\end{proof}

With these results, we may now present a total complexity theorem for Algorithm~\ref{alg:powell}. Recall that Theorem \ref{thm:tr_complexity_ef} provides a bound for $|{\cal S}_\epsilon|+ |{\cal U}_\epsilon|$. As we have just demonstrated 
the number of consecutive model improving iterations cannot be larger than $2 n\log n + 4n + n|\log\log(\Lambda)|$ until set ${\cal Y}_k$ is $\Lambda$-poised in $B(x_k,\Delta_k)$.  Let $\Lambda=1+\frac{1}{n}$, then the number of consecutive iterations can be bounded by ${\cal O}( n\log n)$. By Corollary \ref{thm:fully_linear_lambda} we have, for any $\Delta_k \geq \sqrt{\frac{4\epsilon_f \Lambda}{L + \kappa_{bhm}}}$,
\begin{align*}
		\kappa_{eg} &= \left(L + \kappa_{bhm}\right) \sqrt{n} \sqrt{n (\Lambda^2 - 1) + 2} =\Theta (\sqrt{n})\\
		\kappa_{ef} &= \kappa_{eg}+\frac{L+\kappa_{bhm}}{2} =\Theta (\sqrt{n}).
	\end{align*}

Thus we have the following corollary.  

\begin{corollary}
	Under the same assumptions as Theorem~\ref{thm:tr_complexity_ef},  for any 
	$\epsilon>\sqrt{\frac{4\epsilon_f}{\gamma^2 \min \{C_2, L+\kappa_{bhm}\} C_1^2}}$ (i.e $\epsilon \geq \Omega (\sqrt{n\epsilon_f})$)
	the total oracle complexity of Algorithm~\ref{alg:powell}
is bounded as 	
	\[
	{\cal C}_\epsilon\le {\cal O} \left( \frac{n^{3/2}\log n}{\epsilon^2} \right)
	\]	
if $\eta_2=\sqrt{n}$ and as 
	\[
		{\cal C}_\epsilon\le  {\cal O} \left( \frac{n^{2}\log n}{\epsilon^2} \right)
	\]	
	if $\eta_2$ is constant. 	
\end{corollary}

%\section{Ensuring fully-linear models via Lagrange polynomials.}
%\label{sec:lagrange}
%\input{lagrange}
%
%\section{Extensions to arbitrary polynomial basis.}
%\label{sec:powell2}
%\input{powell_2.tex}
%

%\section{Self-Correcting Method.}% and their behavior in individual iterations} 
%\label{sec:self-correct}
%\input{self-correct.tex}

\section{Model based trust region methods in subspaces.}
\label{sec:subspace_ffd}

We now consider a trust region method where a model $m(x)$ is built and optimized in a random low-dimensional subspace of ${\mathbb R}^n$. The idea of using  random subspace embeddings 
within derivative-free methods has gained a lot of popularity in the literature lately. It was shown in \cite{gratton2018complexity} that applying direct search methods in a low-dimensional subspace reduces the oracle complexity from ${\cal O}(n^2\epsilon^{-2})$ to ${\cal O}(n\epsilon^{-2})$ (with dependence on the subspace dimensions suppressed). 
A random subspace version of a model-based TR method was first studied in \cite{Cartis2023} with the use of Johnson-Lindenstrauss (JL) subspace embeddings which achieved complexity ${\cal O}(n^2\epsilon^{-2})$. Later in \cite{dzahini2024stochastic} this approach was combined with a stochastic model-based trust region method.  Recently it was shown in \cite{ChaudhryScheinberg2026ICM} that   a subspace model-based trust region method achieves an improved ${\cal O}(n\epsilon^{-2})$ complexity when the model is based on a random projection rather than a JL embedding. The key difference lies in the scaling of the projected gradient that is being estimated. A recent note \cite{cartis2026note} confirms that by rescaling the JL transformation  in   \cite{Cartis2023}  the  rate  ${\cal O}(n\epsilon^{-2})$ can be achieved.

The works  \cite{Cartis2023, ChaudhryScheinberg2026ICM} do not consider noisy function values. In \cite{dzahini2024stochastic} the noise in the function values is stochastic and is assumed to be reducible to any desired accuracy, dictated by the trust region radius, which is allowed to 
shrink to an arbitrarily small value.  Here we extend the analysis of  a subspace trust region method from \cite{ChaudhryScheinberg2026ICM} to accommodate fixed (deterministic) noise in the function oracle. The fundamental difficulty of doing so is that the algorithmic framework with random subspaces does not by itself guarantee a lower bound on the trust region radius, which is necessary for the  analysis of the noisy function oracles, as we have seen in the sections above. As a consequence we need to introduce two algorithmic modifications - a relaxed step acceptance criterion and an enforced lower bound on the trust region radius.

We begin by introducing the subspace embedding of our problem. 
Given a matrix $Q\in {\mathbb R}^{n\times q}$, with $q\leq n$ and orthonormal columns, $QQ^T\nabla \phi(x)$ is an orthogonal projection of $\nabla \phi(x)$ onto a subspace spanned by the columns of $Q$ (we will call it a subspace induced by $Q$).  We also define a reduction of  $\phi(x)$  to the subspace, given by $Q$ around $x$: ${\hat \phi}(v)={\phi}(x+Qv)$, $v\in {\mathbb R}^{q}$, which implies $Q\nabla {\hat \phi}(0)=QQ^T\nabla {\phi}(x)$. Similarly we define 
${\hat m}(v)={m}(x+Qv)$,  $v\in {\mathbb R}^{q}$, which implies $Q\nabla {\hat m}(0)=QQ^T \nabla {m}(x)$. 

We now present a modified trust-region algorithm that constructs models and computes steps in the subspace. 
At each iteration  $k \in \{0,1,\dots\}$ the algorithm  chooses  $Q_k\in {\mathbb R}^{n\times q}$ with orthonormal columns. The model $m_k$ is defined as 
\begin{equation}\label{eq:model_def_sub}
	m_k(x_k+Q_kv) = {\phi}(x_k) + g_k^TQ_kv  + \frac{1}{2} v^TQ_k^T H_k Q_kv. 
\end{equation}
For any vector $v$, $ g_k^T Q_kv = g_k^TQ_kQ_k^T Q_kv $, thus without loss of generality, we will assume that $Q_kQ_k^Tg_k=g_k$, in other words, $g_k$ lies in the 
subspace induced by  $Q_k$. 
We define the trust region in the subspace induced by $Q_k$ as  $B_{Q_k}(x_k, \Delta_k)=\{z:\, z=x_k+Q_kv, \ \|v\|\leq \Delta_k\}$.

We will also need  the definition of  a fully linear model  with respect to the subspace. 
% This definition is a modification of a regular fully linear model definition in two ways - the bounds on the function and gradient approximations only apply in the subspace and allow for some constant error term.
We use the following definition which is the same as in \cite{ChaudhryScheinberg2026ICM}. 

% \begin{definition}[Fully-linear model in a subspace] \label{def:fully-linear-subspace}
% Given a matrix with orthonormal columns $Q\in {\mathbb R}^{n\times q}$, let  $B_Q(x, \Delta)=\{z:\, z=x+Qv, \ \|v\|\leq \Delta\}$. 
% Let
%  \begin{equation}\label{eq:model_def_sub_nok}
% 	m(x+Qv) = {\phi}(x) + g^TQv  + \frac{1}{2} v^TQ^T H Qv
% \end{equation}
% and ${\hat m}(v)={m}(x+Qv)$,  $v\in {\mathbb R}^{q}$. 
% We say that model $m(x+s)$ is $\epsilon_{ef},\epsilon_{eg}$-approximately $\kappa_{ef},\kappa_{eg}$-fully linear model of $\phi(x+s)$  on 
%  $B_Q(x, \Delta)$ if 
% \begin{equation}\label{eq:fully-lin-sub-kappaeg}
% \|\nabla \hat m(0)-\nabla \hat \phi (0)\|\leq \kappa_{eg}\Delta + \epsilon_{eg}.
% \end{equation}
% and
% \begin{equation}\label{eq:fully-lin-sub-kappaef}
% |  m(x+Qv)- \phi (x+Qv)|\leq \kappa_{ef}\Delta^2 + \epsilon_{ef}.
% \end{equation}
% \end{definition}

\begin{definition}[Fully-linear model in a subspace] \label{def:fully-linear-subspace}
Given a matrix with orthonormal columns $Q\in {\mathbb R}^{n\times q}$, let  $B_Q(x, \Delta)=\{z:\, z=x+Qv, \ \|v\|\leq \Delta\}$. 
Let
 \begin{equation}\label{eq:model_def_sub_nok}
	m(x+Qv) = {\phi}(x) + g^TQv  + \frac{1}{2} v^TQ^T H Qv
\end{equation}
and ${\hat m}(v)={m}(x+Qv)$,  $v\in {\mathbb R}^{q}$. 
We say that model $m(x+Qv)$ is a $\kappa_{ef}, \kappa_{eg}$-fully linear model of $\phi(x+Qv)$  on 
 $B_Q(x, \Delta)$ if 
\begin{equation}\label{eq:fully-lin-sub-kappaeg}
\|\nabla \hat m(0)-\nabla \hat \phi (0)\|\leq \kappa_{eg}\Delta
\end{equation}
and 
\begin{equation}\label{eq:fully-lin-sub-kappaef}
|\hat m(v)-\hat \phi (v)|\leq \kappa_{ef}\Delta^2
\end{equation}
for all $\|v\|\leq \Delta$. 
\end{definition}

Here too \eqref{eq:fully-lin-sub-kappaeg} implies \eqref{eq:fully-lin-sub-kappaef} 
with a specific value of $\kappa_{ef}$.

\begin{lemma}[Lemma~6.6 from \cite{ChaudhryScheinberg2026ICM}]\label{lem:subspace_error}
Under Assumptions \ref{assum:lip_cont} and \ref{assum:tr}, if \eqref{eq:fully-lin-sub-kappaeg} holds then
$m$ is $\kappa_{ef}, \kappa_{eg}$-fully linear on  $B_{Q}(x, \Delta)$ with
\[
\kappa_{ef}= \kappa_{eg}+\frac{L_Q+\kappa_{bhm}}{2}
\] 
where $L_Q$ is the Lipschitz constant of $QQ^T\nabla \phi (x)$. 
\end{lemma}
%The trial step is computed as 
%\begin{equation}\label{eq:sub_sk}
%s_k= \approx \arg\min_s \{m_k(x_k+Q_ks):~s \in B(0,\Delta_k)\}$.
%\end{equation}

\begin{algorithm}[!ht] 
    \caption{~\textbf{Trust region method based on fully-linear models in subspace}}
    \label{alg:tr_sub}
       {\bf Inputs:} Inexact zeroth order oracle $|f(x)-\phi(x)|\leq \epsilon_f$, minimum radius $\Delta_{\min}$, initial  $x_0$, $\Delta_0 \geq \Delta_{\min}$, initial matrix $Q_0\in {\mathbb R}^{n\times q}$ with orthonormal columns and  $\eta_1\in(0,1)$, $\eta_2 > 0$, and $\gamma\in(0,1)$.  \\
      \For{$k=0,1,2,\cdots$}{
           \nl  For the current $Q_k$ compute model $m_k$ as in \eqref{eq:model_def_sub}. \\
        \nl Compute a trial step $x_k+s_k$  where $s_k=Q_kv_k$ with $v_k\approx \arg\min_v \{m_k(x_k+Q_kv):~\|v\|\leq \Delta_k\}$.\\
        \nl Compute the ratio $\rho_k$ as \[ \rho_k= \frac{{f}(x_k) - {f}(x_k+s_k) + 2\epsilon_f}{m_k(x_k) - m_k(x_k+s_k)}.
        \] \\
        			\nl Update the iterate and the TR radius as 
		\[  (x_{k+1}, \Delta_{k+1}) \gets \left\{ \begin{aligned} 
			&(x_k+s_k, \gamma^{-1} \Delta_k) &&\text{if } \rho_k \ge \eta_1 \text{ and } \|g_k\| \ge \eta_2 \Delta_k, \\
			&(x_k, \Delta_k) && \text{else, if the model is not fully-linear in } B_{Q_k}(x_k,\Delta_k).  \\
			&(x_k, \max\{\gamma\Delta_k, \Delta_{min}\}) &&\text{otherwise. }
		\end{aligned} \right. 
		\] \\
	\nl Update the subspace  
        \[  
        Q_{k+1} \in {\mathbb R}^{n\times q} \gets \left\{ \begin{aligned} 
        	& \text{random} &&\text{if } \rho_k \ge \eta_1 \text{ and } \|g_k\| \ge \eta_2 \Delta_k \text{ or if a model is fully-linear }\\
        	&Q_k &&\text{otherwise. }
        \end{aligned} \right. 
        \] \\
  \nl Perform some model improvement steps.    }
\end{algorithm}

We will assume, as before,  that  Assumption \ref{assum:tr} holds. The two key modifications of Algorithm \ref{alg:tr_sub} compared to Algorithm \ref{alg:tr} are the additional term $2\epsilon_f$ in the definition of $\rho$ and the imposed lower bound on  the trust region radius $\Delta_{min}$. Both of these modifications are needed because the noisy zeroth-order oracle makes it essential for $\Delta_k$ to remain sufficiently positive  to ensure fully linear models. Such a lower bound on $\Delta_k$ is ensured in the deterministic framework of Algorithm \ref{alg:tr} for sufficiently large $\epsilon$. But this is not the case when random subspaces are used because unsuccessful iterations can occur even if $\Delta_k$ is small, due to the subspace not being chosen well. We note that since $\epsilon_f$ is an upper bound on the error in the zeroth order oracle, any upper estimate of it can be used in the algorithm. Of course, unnecessarily  large values will have an adverse effect on the resulting best achievable accuracy $\epsilon$. Similarly, we will see that the choice of $\epsilon_f$ dictates the best choice for $\Delta_{min}$.

% \begin{proof}
% Observe that if $x+s \in B_{Q}(x,\Delta)$, then we may choose a vector $v$ such that $x+s=x+Qv$ and $\|v\|\leq \Delta$.
% By the triangle inequality, assumptions \ref{assum:lip_cont} and \ref{assum:tr}, and the definition of fully-linear,
% \begin{align*}
% &|m(x+s) - \phi (x+s)| \\
% & = |m(x + Qv) - \phi(x + Qv)| \\
% & = |\hat \phi(v) - \hat \phi(0) - \langle \nabla \hat m(0), v \rangle - \langle H_k v, v \rangle / 2| \\
% & \leq |\hat \phi(v) - \hat \phi(0) - \langle \nabla \hat \phi (0), v \rangle| + |\langle \nabla \hat \phi (0), v \rangle - \langle \nabla \hat m(0), v \rangle| + |\langle H_k v, v \rangle / 2| \\
% & \leq L \|v\|^2/2 + \| \nabla \hat \phi (0) - \nabla \hat m(0) \| \|v\| + \kappa_{bhm} \|v\|^2/2 \\
% & \leq L \Delta^2/2 + (\kappa_{eg}\Delta + \epsilon_{eg} ) \Delta + \kappa_{bhm} \Delta^2 /2 \\
% & =  (L + \kappa_{bhm} + 2 \kappa_{eg}) \Delta^2/2 + \epsilon_{eg} \Delta.
% \end{align*}
% \end{proof}

% In what follows, for the brevity of notation we will refer to models that satisfy  \eqref{eq:fully-lin-sub-kappaeg} as 
% $\epsilon_g$-approximately $\kappa_{eg}$-fully linear on  $B_{Q_k}(x, \Delta)$ with the understanding that 
% Lemma \ref{lem:subspace_error} holds\footnote{We switch notation from $\epsilon_{eg}$ to $\epsilon_g$ for consistency with  \cite{cao2023first}}. 

We now introduce a definition from \cite{ChaudhryScheinberg2026ICM} of a measure of how well the subspace induced by $Q$ aligns with the current gradient.

\begin{definition}[Well aligned subspace] \label{def:well-aligned-subspace}
The subspace spanned by columns of $Q$ is $\kappa_g$-well aligned with  $\nabla \phi(x)$ for a given  $x$ if
      \begin{equation}\label{eq:subspace_req_leq}
\|QQ^T\nabla \phi(x)-\nabla \phi(x) \|\leq \kappa_g\|\nabla \phi(x) \|
\end{equation}    
for some $\kappa_g\in [0, 1)$. 
\end{definition}

While condition \eqref{eq:subspace_req_leq} involves the gradient $\|\nabla \phi(x) \|$ it ultimately reduces to the properties of the subspace. Essentially, it requires that the gradient is not too close to being orthogonal to the subspace induced by $Q$. Similar conditions and terminology have been used in \cite{Cartis2023, ChaudhryScheinberg2026ICM}.

We also recall the following related lemma (recalling that $g_k=Q_kQ_k^Tg_k$). 

\begin{lemma}[Lemma~6.3 from \cite{ChaudhryScheinberg2026ICM}]\label{lem:grad_error_sub}
On iteration $k$, $Q_k$  is $\kappa_g$-well aligned with  $\nabla \phi(x_k)$ if and only if 
  \begin{equation}\label{eq:subspace_req}
\|Q_kQ_k^T\nabla \phi(x_k) \|^2\geq (1-\kappa_{g}^2)\|\nabla \phi(x_k) \|^2.
\end{equation} 
 Also, if $m(x_k+s)$ is $\kappa_{ef}, \kappa_{eg}$-fully linear model of $\phi(x_k+s)$ on  $B_{Q_k}(x_k, \Delta_k)$ then 
  \begin{equation}\label{eq:grad_error_sub}
 \|g_k-Q_kQ_k^T\nabla\phi(x_k)\|\leq \kappa_{eg}\Delta_k 
 \end{equation}
\end{lemma}

We can now show the following lemma which is analogous to Lemma~6.4 from \cite{ChaudhryScheinberg2026ICM}.

\begin{lemma}[sufficient condition for a successful step] \label{lem:tr_sub_success}
  Under Assumptions~\ref{assum:lip_cont} and \ref{assum:tr}, if $Q_k$ is $\kappa_g$-well aligned with $\nabla \phi(x_k)$, $m_k(x_k+s)$ is a $\kappa_{ef}, \kappa_{eg}$-fully linear model of $\phi(x_k+s)$ on $B_{Q_k}(x_k,\Delta_k)$  and if 
    \begin{equation} \label{eq:tr_sub_success_1}
        \Delta_k \le \sqrt{1-\kappa_g^2} \tilde C_1 \|\nabla \phi(x_k)\| 
    \end{equation}
   where
   \[
     \tilde C_1 = (\max\left\{\eta_2,\ \kappa_{bhm},\ \frac{2\kappa_{ef}}{(1-\eta_1)\kappa_{fcd}}\right\}+\kappa_{eg})^{-1}
    \]
    then $\rho_k \ge \eta_1$, $\|g_k\| \ge \eta_2 \Delta_k$, and $x_{k+1} = x_k+s_k$, i.e. the iteration $k$ is successful. 
\end{lemma}

\begin{proof}
    Due to Lemma \ref{lem:grad_error_sub}, specifically \eqref{eq:grad_error_sub}   by triangle inequality, 
    \[ \|Q_kQ_k^T\nabla \phi(x_k)\| \le \|g_k\| + \kappa_{eg}\Delta_k
    \]
    and also due to Lemma \ref{lem:grad_error_sub}
    \begin{equation}\label{eq:tr_sub_success} 
      \Delta_k \le \sqrt{1-\kappa_g^2}\tilde C_1 \|\nabla \phi(x_k)\| \leq \tilde C_1 \|Q_kQ_k^T\nabla \phi(x_k)\|
    \end{equation}
      
    By \eqref{eq:tr_sub_success} we have  
    \[ \begin{aligned}
       (\max\{\kappa_{bhm},\eta_2, \frac{2\kappa_{ef}}{(1-\eta_1)\kappa_{fcd}}\} + \kappa_{eg}) \Delta_k & \le 
         (\|g_k\| + \kappa_{eg}\Delta_k) 
    \end{aligned} \] 
    which implies
    \[
           \max\{\kappa_{bhm},\eta_2\} \Delta_k \le \|g_k\|. 
           \]
    This establishes that $\|g_k\| \ge \eta_2\Delta_k$ and also $m_k(x_k) - m_k(x_k+s_k) \ge \kappa_{fcd} \|g_k\| \Delta_k / 2$ by Assumption~\ref{assum:tr}. 
    Then, using  the fact that $|f(x) - \phi(x)| \leq \epsilon_f$ and the fully linear assumption on  $m_k$ in the subspace induced by $Q_k$ and recalling that $s_k=Q_kv_k$ we have 
    \[ \begin{aligned}
        \rho_k &= \frac{m_k(x_k) - m_k(x_k+s_k) + (f(x_k) - m_k(x_k)) - (f(x_k+s_k) - m_k(x_k+s_k)) + 2\epsilon_f}{m_k(x_k) - m(x_k+s_k)} \\
         &\geq \frac{m_k(x_k) - m_k(x_k+s_k) + (\phi(x_k) - m_k(x_k)) - (\phi(x_k+s_k) - m_k(x_k+s_k))}{m_k(x_k) - m_k(x_k+s_k)}\\
        &\ge 1 - \frac{\kappa_{ef} \Delta_k^2}{m_k(x_k) - m_k(x_k+s_k)}   \ge 1 - \frac{\kappa_{ef} \Delta_k^2}{\kappa_{fcd} \|g_k\| \Delta_k / 2}\\
        & \ge 1 - \frac{2\kappa_{ef}\Delta_k}{\kappa_{fcd} ( \|Q_kQ_k^T\nabla \phi(x_k)\|- \kappa_{eg}\Delta_k)} \ge \eta_1, 
    \end{aligned} \]
    where the last step is true because $ \|Q_kQ_k^T\nabla \phi(x_k)\| \ge \big(\frac{2\kappa_{ef}}{(1-\eta_1)\kappa_{fcd}} + \kappa_{eg}\big) \Delta_k$  follows from  \eqref{eq:tr_sub_success}. 
\end{proof}

Now we show a lower bound on progress made in successful iterations similar to Lemma~4.3 of \cite{cao2023first}.

\begin{lemma}[Progress made in a successful iteration]\label{lem:subspace_prog}
Under Assumptions \ref{assum:lip_cont} and \ref{assum:tr}, in Algorithm~\ref{alg:tr_sub}, if $\rho_k \geq \eta_1$ and $\|g_k\| \ge \eta_2 \Delta_k$, then
\[
\phi(x_k) - \phi(x_{k+1}) \geq C_2 \Delta_k^2 - 4 \epsilon_f,
\]
where $C_2 = \frac{1}{2} \eta_1 \eta_2 \kappa_{fcd} \min \left\{ \frac{\eta_2}{\kappa_{bhm}},1\right\}$.
\end{lemma}
\begin{proof}
Since $\rho_k \geq \eta_1$, we have
\[
\eta_1 \leq \frac{f(x_k) - f(x_k+s_k) + 2 \epsilon_f}{m_k(x_k) - m_k(x_k + s_k)} \leq \frac{\phi(x_k) - \phi(x_k+s_k) + 4 \epsilon_f}{m_k(x_k) - m_k(x_k + s_k)},
\]
which we can rearrange as $\phi(x_k) - \phi(x_k+s_k) \geq \eta_1(m_k(x_k) - m_k(x_k + s_k)) - 4 \epsilon_f$.
Since $\|g_k\| \geq \eta_2 \Delta_k$, then by \eqref{eq:Cauchy decrease}, we have
\[
\eta_1(m_k(x_k) - m_k(x_k + s_k)) \geq \frac{\eta_1 \kappa_{fcd}}{2} \|g_k\| \min \left\{ \frac{\|g_k\| }{\kappa_{bhm}}, \Delta_k \right\} \geq C_2 \Delta_k^2.
\]
\end{proof}

% \begin{lemma}\label{eq:tr_sub_success_1}
% 	Under assumptions \ref{assum:lip_cont} and \ref{assum:tr}, if $m_k$ is $\kappa_{eg}, \epsilon_g$-approximately fully linear, if $Q_k$ is $\kappa_g$-well aligned with $\nabla \phi(x_k)$,  and if 
% 	\[
% 	\Delta_k \leq \hat{C}_1 \| \phi (x_k) \| - C_1'' \epsilon_g
% 	\]
% 	where $\hat{C}_1 = \sqrt{1-\kappa_g^2}C_1'$ and $C_1',C_1''$ are as defined in Lemma~\ref{lem:subspace_succ},
% 	then $\rho_k \geq \eta_1$ and $\|g_k\| \ge \eta_2 \Delta_k$ in Algorithm~\ref{alg:tr_sub}.
% \end{lemma}

% \begin{proof}
% By the orthogonality of the columns of $Q$, we have that $\|QQ^T\nabla \phi(x_k)-\nabla \phi(x_k) \|\leq \kappa_g\|\nabla \phi(x_k) \|$ implies that $\|QQ^T\nabla \phi(x_k) \| \geq \sqrt{1 - \kappa_{g}^2} \|\nabla \phi(x_k)\|$.
% With this, and recalling that $\nabla \hat \phi (0) = QQ^T\nabla \phi(x_k)$, the result follows from Lemma~\ref{lem:subspace_succ}.
% \end{proof}

From these results one could bound the number of needed iterations if one had a guarantee of having well-aligned subspaces, however we wish to cover the random case where the subspace is well-aligned only with some probability.
Accordingly, we study the random subspace case next.

\subsection{Complexity analysis under random subspace selection.}
Observe that the matrix $Q_k$, and the trust region radius, $\Delta_k$ do not change on model improving iterations.
We define $Q_t$ to be the $t$th random matrix $Q_k$ and similarly $\Delta_t$ and $x_t$ are the trust region radius and iterate which correspond to $Q_t$. In other words, index $t$ counts iterations that follow either a successful or an unsuccessful iteration. 
In what follows we will essentially derive a bound for $|{\cal S}_\epsilon|+ |{\cal U}_\epsilon|$ by examining what happens to the objective function and trust region radius after these iterations (since we know that after model improving iterations neither the TR radius nor the function value change).

We now define some relevant stochastic processes:
\begin{align*}
	I_t&=\mathbbm{1}\{Q_t\ \text{ is\ } \kappa_g \text{-well\ aligned\ with\ }  \nabla \phi(x_t)\},\\
	A_t&=\mathbbm{1}\{ Q_t {\rm \ leads\ to\ a\ successful\ iteration\ i.e.,\ } \Delta_{t+1}=\gamma^{-1}\Delta_t \},\\
	B_t&=\mathbbm{1}\{ \Delta_t >\hat C_1 \|\nabla \phi(x_t)\| \} , 
\end{align*}
 where $\hat C_1 = \sqrt{1-\kappa_g^2}\tilde  C_1$ for $\tilde C_1$ as defined in Lemma~\ref{lem:tr_sub_success}.

We will say that matrix $Q_t$ is "true" if $I_t=1$.
Let $T_\epsilon$ be the first $t$ such that the iterate $x_t$ produced by Algorithm \ref{alg:tr_sub} achieves $\|\nabla \phi(x_t)\|\leq \epsilon$.
Let $\mathcal{F}_{t-1}$ denote the $\sigma$-algebra generated by the first $t$ matrices, $\mathcal{F}_{t-1}=\sigma\left( Q_0, Q_1, \ldots Q_{t-1}\right )$. 
We note that the random variables $x_t$ and $\Delta_t$ are measurable with respect to  $\mathcal{F}_{t-1}$.
We define $m_t$, $s_t$, $\rho_t$ to be the last model, proposed step, and ratio, respectively, which correspond to the matrix $Q_t$.
The random variables $m_t$, $s_t$ and $\rho_t$ are measurable with respect to  $\mathcal{F}_{t}$. 
The random variable $T_\epsilon =\min\{ t:\ \|\nabla \phi(x_t)\|\leq \epsilon\}$ is a stopping time adapted to the filtration $\{\mathcal{F}_{t-1}\}$.

\begin{assumption}\label{ass:prob_true_iter}
	There exists a $\theta\in (\frac{1}{2}, 1]$ such that
	\[
	{\mathbb P}\{I_t=1| {\mathcal F}_{t-1}\}\geq \theta. 
	\]
\end{assumption}

By Lemma~6.7 of \cite{ChaudhryScheinberg2026ICM}, if we take $Q_t$ such that the subspace it induces is uniformly distributed, we can take $\theta\geq \frac{243}{443}>1/2$ and $\kappa_g=\sqrt{1-\frac{q}{10n}}$.

Note that $\sigma(B_t)\subset {\cal F}_{t-1}$ and $\sigma(A_t)\subset {\cal F}_{t}$, that is the random variable 
$B_t$ is fully determined by matrices $Q_0, \ldots Q_{t-1}$ produced by the algorithm, while $A_t$ is fully determined by the  matrices $Q_0, \ldots Q_{t}$.  
The stochastic process described here has essentially the  same dynamics as the process analyzed in \cite{cartis2018global} and  \cite{ChaudhryScheinberg2026ICM} enabling us to reuse the results. The only differences are the presence of lower bound $\Delta_{min}$ and the possible increase of $\phi(x_k)$ on some successful iterations. The lower bound does not alter the main properties of the dynamics of $\Delta_k$, since by fixing $\epsilon$ to be sufficiently large with respect to $\Delta_{min}$ we ensure that $\Delta_{\min}<\hat C_1 \|\nabla \phi(x_t)\|$ for $t=0, 1, \ldots T_{\epsilon}-1$. Thus we retain the key property which follows from  Lemma \ref{lem:tr_sub_success}: 
\[
A_t\geq I_t(1-B_t), 
\]
in other words, if matrix $Q_t$ is true and the trust region radius is sufficiently small, then the iteration is successful. 

The increase of objective function on certain iterations is due to the relaxed definition of $\rho_k$ and the error in the zeroth order oracle. Such situations have been previously analyzed for line (step) search in \cite{berahas2021global,jin2021high}
and trust-region method in \cite{cao2023first}. The analysis here is simpler but the key idea is that the increase is bounded by $4\epsilon_f$ and occurs on iterations whose number is not too large compared to the number of iterations where function decreases. By ensuring that the decrease is  sufficiently large to compensate for the increase, the results are derived. 
Below we present the analysis.

To bound the total number of successful and unsuccessful iterations  we first bound the number of matrices that lead to successful iterations with large $\Delta$.
For that let $\bar B_t=\mathbbm{1}\{ \Delta_t \geq \gamma \hat C_1 \epsilon \}$ (note that $B_t=1\Rightarrow \bar B_1=1$).
Then from the dynamics of Algorithm~\ref{alg:tr_sub} we have the bound similar to \cite{cartis2018global} (also used in \cite{ChaudhryScheinberg2026ICM}). 

\begin{lemma}\label{lem:bound_on_big2}
	Suppose $\epsilon > \sqrt{\frac{8\epsilon_f}{C_2 \gamma^2 \hat{C}_1^2 }}$
	For any $l\in \{0,\ldots,T_{\epsilon}-1\}$ and for all realizations of Algorithm \ref{alg:tr_sub}, we have 
	\[
	\sum_{t=0}^l  \bar B_tI_tA_t \leq \sum_{t=0}^l  \bar B_t A_t \leq \frac{\phi(x_0)-\phi^\star + 4\epsilon_f(\sum_{t=0}^l  (1 - \bar B_t) A_t)}{\frac{1}{2}C_2(\gamma (\hat C_1\epsilon))^2},
	\]
\end{lemma}
\begin{proof}
Since 	$\epsilon > \sqrt{\frac{8\epsilon_f}{C_2 \gamma^2 \hat{C}_1^2 }}$, $\Delta_t \geq \gamma \hat C_1 \epsilon$ implies that $\Delta_t \geq \sqrt{\frac{8\epsilon_f}{C_2}}$ which in turn implies that for large successful iterations (corresponding to $\bar B_t A_t$), $\phi(x_t) - \phi(x_{t+1}) \geq \frac{1}{2}C_2 \Delta_t^2$ by Lemma~\ref{lem:subspace_prog}.
For small successful iterations (corresponding to $(1 - \bar B_t)A_t$), by the same lemma, we have $\phi(x_t) - \phi(x_{t+1}) \geq -4 \epsilon_f$.
The result follows.
\end{proof}

A useful lemma that easily follows from the dynamics is as follows. 
\begin{lemma}\label{lem:bound_on_big}
	Suppose $\Delta_0 \geq \hat C_1 \epsilon$ and $\Delta_{\min} \leq \gamma \hat C_1 \epsilon$.
	For any $l\in \{0,\ldots,T_{\epsilon}-1\}$ and for all realizations of Algorithm \ref{alg:tr_sub}, we have 
	\[
	\sum_{t=0}^l B_t(1- A_t) \leq \sum_{t=0}^l\bar B_tA_t + \log_{\gamma}\left (\frac{\hat C_1 \epsilon}{\Delta_0}\right ). 
	\]
\end{lemma}

The following result is shown in \cite{cartis2018global}  under Assumption \ref{ass:prob_true_iter}, 
\[
\mathbb{E}\left (\sum_{t=0}^{T_\epsilon-1} \bar B_t(1-I_t)\right ) \leq \frac {1-\theta}{\theta } \mathbb{E}\left(\sum_{t=0}^{T_\epsilon-1} \bar B_tI_t\right ),
\]
from which the following lemma is derived. 
\begin{lemma}\label{lem:hittime2}
Let Assumption~\ref{ass:prob_true_iter} hold.
	Under the condition that $\theta >1/2$, $\Delta_0 \geq \hat C_1 \epsilon$, and $\Delta_{\min} \leq \gamma \hat C_1 \epsilon$, we have
	\[
	\mathbb{E}\left (\sum_{t=0}^{T_\epsilon-1}  B_t \right )\leq \frac{1}{2\theta-1}\left (\sum_{t=0}^{T_\epsilon-1}  \bar B_t A_t + \log_{\gamma}\left (\frac{\hat C_1 \epsilon}{\Delta_0}\right )\right ) .
	\]
\end{lemma}

Finally the following lemma is shown in \cite{cartis2018global} for the stochastic processes $I_t,A_t$ and $B_t$ since $A_t\geq I_t(1-B_t)$ and by the dynamics of $\Delta_t$.

\begin{lemma}\label{lem:hittime1}
Let Assumption~\ref{ass:prob_true_iter} hold.
	\[
	\mathbb{E}\left (\sum_{t=0}^{T_\epsilon-1}  (1-B_t) \right)\leq \frac{1}{2\theta }{\mathbb E}\left [ T_\epsilon\right ].
	\]
\end{lemma}

Putting these lemmas together we obtain the final expected complexity result. 

\begin{theorem}\label{thm:tr_sub_complexity}
Let Assumption~\ref{assum:lip_cont}, Assumption \ref{assum:tr} and Assumption \ref{ass:prob_true_iter} hold.
%Assume that for all $t$, such that $\Delta_{t+1} = \gamma \Delta_t$, $m_t(x_t+s)$ is a $\kappa_{ef}, \kappa_{eg}$-fully linear model of $\phi(x_t+s)$ on $B_{Q_t}(x_t,\Delta_t)$. \ks{Why do we need to assume this? it is prescribed the algorithm.}
Then for any $\epsilon> \sqrt{\frac{16\epsilon_f}{(2\theta-1)^2 C_2 \hat C_1^2 \gamma^2}}$, assuming an initial trust-region radius $\Delta_0 \geq \hat C_1 \epsilon$, and $\Delta_{\min} \leq \gamma \hat C_1 \epsilon$,
	let   $T_\epsilon $ be the random stopping time for the event $\{\|\nabla \phi(x_t)\|\leq \epsilon\}$. We have the bound 
	\[
	{\mathbb E}\left [ T_\epsilon\right ]\leq \frac{4 \theta}{(2\theta-1)^2}\left (\frac{\phi(x_0)-\phi^\star}{\frac{1}{2}C_2(\gamma \hat C_1\epsilon)^2} + \log_{\gamma}\left (\frac{\hat C_1 \epsilon}{\Delta_0}\right )\right )
	\]
	where  $\hat C_1 = \sqrt{1- \kappa_g^2}\tilde C_1$ for $\tilde C_1$ as in \eqref{eq:tr_sub_success_1} and $C_2$ as in Lemma \ref{lem:subspace_prog}.  
\end{theorem}
\begin{proof}
Observe that $\sum_{t=0}^{T_\epsilon -1}  (1 - \bar B_t) A_t \leq \sum_{t=0}^{T_\epsilon-1}  (1 - B_t)$.
Thus, from Lemmas~\ref{lem:bound_on_big2} and \ref{lem:hittime2}, we have
\[
\mathbb{E}\left (\sum_{t=0}^{T_\epsilon-1}  B_t \right )\leq \frac{1}{2\theta-1}\left (\frac{\phi(x_0)-\phi^\star}{\frac{1}{2}C_2(\gamma \hat C_1\epsilon)^2} + \log_{\gamma}\left (\frac{\hat C_1 \epsilon}{\Delta_0}\right )\right ) + C_3 \left(\sum_{t=0}^{T_\epsilon-1}  (1 - B_t)\right),
\]
where $C_3 = \frac{4 \epsilon_f}{(2\theta-1)\frac{1}{2}C_2(\gamma \hat C_1\epsilon)^2}$.
It follows that 
\[
{\mathbb E}\left [ T_\epsilon\right ] \leq \frac{1}{2\theta-1}\left (\frac{\phi(x_0)-\phi^\star}{\frac{1}{2}C_2(\gamma \hat C_1\epsilon)^2} + \log_{\gamma}\left (\frac{\hat C_1 \epsilon}{\Delta_0}\right )\right ) + (1+C_3)\mathbb{E}\left (\sum_{t=0}^{T_\epsilon-1}  (1 - B_t) \right ).
\]
Combining with Lemma~\ref{lem:hittime1}, we obtain
\[
{\mathbb E}\left [ T_\epsilon\right ] \leq \frac{1}{2\theta-1}\left (\frac{\phi(x_0)-\phi^\star}{\frac{1}{2}C_2(\gamma (\hat C_1\epsilon)^2} + \log_{\gamma}\left (\frac{\hat C_1 \epsilon}{\Delta_0}\right )\right ) + (1+C_3)\frac{1}{2\theta }{\mathbb E}\left [ T_\epsilon\right ].
\]
The lower bound on $\epsilon$ implies that $C_3 \leq \theta - \frac{1}{2}$ and so $\frac{1+C_3}{2\theta} \leq \frac{2\theta+1}{4\theta}$.
Thus we have
\[
\frac{2\theta-1}{4\theta}{\mathbb E}\left [ T_\epsilon\right ] \leq \frac{1}{2\theta-1}\left (\frac{\phi(x_0)-\phi^\star}{\frac{1}{2}C_2(\gamma \hat C_1\epsilon)^2} + \log_{\gamma}\left (\frac{\hat C_1 \epsilon}{\Delta_0}\right )\right ).
\]
\end{proof}

In order to use this theorem for effective complexity bounds, we must specify how to form models in a subspace.
In the following subsections we discuss the two different approaches we used in the full space case - finite differences and interpolation based 
$\Lambda$-poised sets. The key difference now is in the lower bound on $\Delta_k$ imposed by $\Delta_{min}$ rather than occurring automatically. 

\subsection{Building models in a subspace}

One can form a gradient estimate via a subspace version of \eqref{eq:finite_diffs}.
This can take the following form given in \cite{ChaudhryScheinberg2026ICM}:
\begin{equation}\label{eq:subspace_ffd}
	\hat g(0)=\sum_{i=1}^q \frac{f(x+\delta Qu_i)-f(x)}{\delta}u_i 
\end{equation}
where $u_i$ is the $i$th column of an orthogonal ${q\times q}$ matrix. Let us define $g(x)=Q\hat g(0)$. 

By similar analysis as in \cite{berahas2021theoretical}, we can derive the bound
\begin{align*}
	\left \| \nabla \hat m(0)-\nabla \hat \phi (0) \right \| \leq \frac{\sqrt{q} L {\delta}}{2} + \frac{2\sqrt{q}  \epsilon_f}{{\delta} }.
\end{align*}
Choosing $\delta = \Delta_k$, we have that $m_k$ is $\kappa_{ef}, \kappa_{eg}$-fully linear model for
\[
	\kappa_{eg} = \frac{\sqrt{q}L}{2} + \frac{2\sqrt{q} \epsilon_f}{\Delta_{\min}^2}, \quad
	\kappa_{ef} = \kappa_{eg}+\frac{L_Q+\kappa_{bhm}}{2}.
\]
In this case there are no model improving iterations. Thus all iterations are either successful or unsuccessful and each iteration 
requires either $q$ or $q+1$ function evaluations.  With these specifics we can give a final complexity bound for Algorithm~\ref{alg:tr_sub}. For simplicity of the presentation we will give the final bounds in terms of the key components, such as $n$, $\epsilon$, $\epsilon_f$, $L$ and $\Delta_{min}$.

\begin{theorem}\label{thm:tr_sub_complexity2}
Let Assumption~\ref{assum:lip_cont} and \ref{assum:tr} hold.
When randomizing, take $Q_t$ such that the subspace it induces is uniformly distributed with $q \geq 3$.
For all $k=0, 1, \ldots K_{\epsilon}-1$, define $m_k(x_k+s)$ with $g_k$ as in \eqref{eq:subspace_ffd}. 
Let the parameters $\eta_1,\eta_2,\kappa_{bhm},\gamma$ be constants and assume $L \geq 1$.
Then for $\epsilon > \Omega(\sqrt{n} (L + \frac{\epsilon_f}{\Delta_{\min}^2}) (\sqrt{\epsilon_f} + \Delta_{\min}))$, assuming an initial trust-region radius $\Delta_0 \geq \Omega(\sqrt{ \epsilon_f}  + \Delta_{\min})$, let   $ K_\epsilon $ be the random stopping time for the event $\{\|\nabla \phi(x_k)\|\leq \epsilon\}$. We have the bound 
	\[
	{\mathbb E}\left [ K_\epsilon\right ]\leq \mathcal{O} \left( \left(\frac{n}{\epsilon^2} \right) \left(L+\frac{\epsilon_f}{\Delta_{\min}^2}\right)^2\right)
	\]
	where the ``big-O'' notation suppresses constant factors and an additive logarithmic term.  
\end{theorem}
\begin{proof}
Since we use \eqref{eq:subspace_ffd} in every iteration, we have that $k$ and $t$ are equivalent.
From the definition of $\tilde C_1$ and the bound on $\kappa_{ef}$, we have $\tilde C_1^{-1} = \Theta(\kappa_{eg}) = \Theta(\sqrt{q} (L + \frac{\epsilon_f}{\Delta_{\min}^2}))$. 
By Lemma~6.7 of \cite{ChaudhryScheinberg2026ICM}, we have $\sqrt{1 - \kappa_g^2} = \Theta(\sqrt{\frac{q}{n}})$.
Thus we have $\hat C_1^{-1} = \Theta(\sqrt{n} (L + \frac{\epsilon_f}{\Delta_{\min}^2}))$.
Then the condition that $\epsilon> \sqrt{\frac{16\epsilon_f}{(2\theta-1)^2 C_2 \hat C_1^2 \gamma^2}}$ from Theorem~\ref{thm:tr_sub_complexity} becomes that $\epsilon> \Omega(\sqrt{n} (L + \frac{\epsilon_f}{\Delta_{\min}^2}) \sqrt{\epsilon_f})$.
The condition that $\Delta_{\min} \leq \gamma \hat C_1 \epsilon$, becomes $\epsilon> \Omega(\sqrt{n} (L + \frac{\epsilon_f}{\Delta_{\min}^2}) \Delta_{\min})$.
Combining these two bounds results in the condition $\epsilon > \Omega(\sqrt{n} (L + \frac{\epsilon_f}{\Delta_{\min}^2}) (\sqrt{\epsilon_f} + \Delta_{\min}))$.
Finally the condition that $\Delta_0 \geq \hat C_1 \epsilon$ becomes $\Delta_0 \geq \mathcal{O}(\sqrt{ \epsilon_f}  + \Delta_{\min})$.
The expected iteration bound then follows directly from Theorem~\ref{thm:tr_sub_complexity}.
\end{proof}

Here we note that the lower bound on $\epsilon$, $\Omega(\sqrt{n} (L + \frac{\epsilon_f}{\Delta_{\min}^2}) (\sqrt{\epsilon_f} + \Delta_{\min}))$, can be approximately optimized by taking $\Delta_{\min} = \Theta(\sqrt{\epsilon_f})$.
The lower bound then becomes $\epsilon \geq \Omega(\sqrt{n \epsilon_f})$ with a rate of 
\[
{\mathbb E}\left [ K_\epsilon\right ]\leq \mathcal{O}  \left(\frac{n}{\epsilon^2} \right).
\]
We also note that each iteration requires only $\mathcal{O}(q)$ function evaluations.
Thus the total expected complexity rate is 
\[
{\mathbb E}\left [ {\cal C}_\epsilon\right ]\leq\mathcal{O}\left (\frac{nq}{\epsilon^2}\right ).
\]

\subsection{Geometry-correcting algorithm in subspaces.} 
We now describe a geometry-correcting version of Algorithm~\ref{alg:tr_sub}. This algorithm performs model improving steps of Algorithm~\ref{alg:powell}  until either successful step is achieved or a fully linear model in the subspace is formed. At that point it terminates the work in that subspace and regenerates a new subspace as well as restarts the models using the initial sample sets ${\cal Y}_0$, ${\cal Z}_0$. This initialization choice is somewhat arbitrary and can be replaced by different initial sets. Each time, however, this requires computation of new function values for all points in the ``initial'' sample set. Our computational results  show that this is quite expensive, if we use ${\cal Y}_0$ and ${\cal Z}_0$ that contain $q$ points in each.  We can delay resampling the random subspace until several successful or unsuccessful steps have been encountered and extend the theory to such strategies. However our computational results so far do not support an advantage of this approach. We can also choose ${\cal Y}_0$ to contain only 1 point and  ${\cal Z}_0$ to be empty by modifying model improvement step and Lagrange polynomial computation to allow for incomplete sets. This modification is simple from the theory point of view but whether it can be practically competitive is yet unclear. Thus we retain the simplest approach for our analysis. 

\begin{algorithm}[!ht] 
    \caption{~\textbf{Geometry-correcting algorithm in subspace}}
    \label{alg:powell_sub}
{\bf Inputs:} Inexact zeroth order oracle $|f(x)-\phi(x)|\leq \epsilon_f$, minimum radius $\Delta_{\min}$, initial  $x_0$, $\Delta_0 \geq \Delta_{\min}$, initial matrix $Q_0\in {\mathbb R}^{n\times q}$ with orthonormal columns and  $\eta_1\in(0,1)$, $\eta_2 > 0,\gamma\in(0,1)$, $\Lambda > 1$, $\Lambda_{sc} \geq 1$.  \\
{\bf Initialization} Initial sets ${\cal Y}_0,{\cal Z}_0 \subset \mathbb{R}^q$  and the function values $f(x_0)$, $f(x_0 + Q_0y_i)$, $y_i \in {\cal Y}_0$, $f(x_0 + Q_0z_i)$, $z_i \in {\cal Z}_0$. A set of Lagrange Polynomials  
$\{\ell_i(v), i=1, \ldots, q\}$ in  ${\cal P}$ for the set ${\cal Y}_0$.\\
      \For{$k=0,1,2,\cdots$}{
           \nl  For the current $Q_k$, build a quadratic model ${\hat m}_k(v)={m}_k(x_k+Q_kv)$ as in \eqref{eq:model_def2} using $f(x_k)$ and $f(x_k + Q_ky_i)$, $y_i\in {\cal Y}_k$, $f(x_k + Q_kz_i)$, $z_i \in {\cal Z}_k$. \\
        \nl Compute a trial step $x_k+s_k$  where $s_k=Q_kv_k$ with $v_k\approx \arg\min_v \{m_k(x_k+Q_kv):~\|v\|\leq \Delta_k\}$.\\
        \nl Compute the ratio $\rho_k$ as \[ \rho_k= \frac{{f}(x_k) - {f}(x_k+s_k) + 2\epsilon_f}{m_k(x_k) - m_k(x_k+s_k)}.
        \] \\
        \nl Update the iterate and the TR radius as 
		\[  (x_{k+1}, \Delta_{k+1}) \gets \left\{ \begin{aligned} 
			&(x_k+s_k, \gamma^{-1} \Delta_k) &&\text{if } \rho_k \ge \eta_1 \text{ and } \|g_k\| \ge \eta_2 \Delta_k, \\
			&(x_k, \Delta_k) && \text{else, if ${\cal Y}_k$ is not $\Lambda$-poised,}  \\
			&(x_k,  \max\{\gamma \Delta_k, \Delta_{min}\}) &&\text{otherwise. }
		\end{aligned} \right. 
		\] \\
	\nl Update the subspace  
        \[  
        Q_{k+1} \in {\mathbb R}^{n\times q} \gets \left\{ \begin{aligned} 
        	& \text{random} && \text{if } \rho \ge \eta_1 \text{ and } \|g_k\| \ge \eta_2 \Delta_k \text{ or if ${\cal Y}_k$ is $\Lambda$-poised }\\
        	&Q_k &&\text{otherwise. }
        \end{aligned} \right. 
        \] \\
    \nl Update the interpolation sets 
        \[  
        {\cal Y}_{k+1},{\cal Z}_{k+1} \gets \left\{ \begin{aligned} 
        	& {\cal Y}_0, {\cal Z}_0  \quad \text{if } \rho \ge \eta_1 \text{ and } \|g_k\| \ge \eta_2 \Delta_k \text{ or if ${\cal Y}_k$ is $\Lambda$-poised }\\
        	&\text{Update as in Step~4 of Algorithm~\ref{alg:powell} } \quad \text{otherwise. }
        \end{aligned} \right. 
        \] \\
}
\end{algorithm}

For this algorithm we can show the following complexity rate.
\begin{corollary}
Under the same assumptions as Theorem~\ref{thm:tr_sub_complexity2}, letting $\Lambda = 1 + \frac{1}{q}$ and $\Delta_{\min} = \sqrt{\epsilon_f}$, and $|{\cal Z}_0|=q$, for any $\epsilon > \Omega(\sqrt{n \epsilon_f})$, the expected total oracle complexity of Algorithm~\ref{alg:powell_sub} is bounded as
	\[
	O \left( \frac{nq \log q}{\epsilon^2}  \right).
	\]  
\end{corollary}
\begin{proof}
By Theorem~\ref{thm:Lambda_to_kappaeg_err}, since we chose $\Lambda = 1 + \frac{1}{q}$, we have the error bound
\[
\left \| \nabla \hat m_k(0)-\nabla \hat \phi (0) \right \| \leq \mathcal{O}\left(\sqrt{q} \left(L \Delta_k + \frac{\epsilon_f}{\Delta_k} \right) \right).
\]
Thus for iterations when ${\cal Y}_k$ is $\Lambda$-poised, we have that $m_k$ is $\kappa_{ef}, \kappa_{eg}$-fully linear model for
\[
	\kappa_{eg}, \kappa_{ef} = \mathcal{O}\left(\sqrt{q} \left(L + \frac{\epsilon_f}{\Delta_{\min}^2} \right) \right).
\]
From this, Theorem~\ref{thm:tr_sub_complexity}, and the arguments from the proof of Theorem~\ref{thm:tr_sub_complexity2}, we can bound the number of non-geometry-correcting iterations (i.e. iterations where $Q_k$ is resampled) by $\mathcal{O}(\frac{n}{\epsilon^2})$.
For these iterations, we use only one oracle call to evaluate $f(x_k+s_k)$, however for the iteration which immediately follows a successful or an unsuccessful iteration, we use $\mathcal{O}(q)$ oracle calls to evaluate $f(x_{k+1}+Q_{k+1}y_i)$ and $f(x_{k+1}+Q_{k+1}z_i)$.
For all other iterations, as in Algorithm~\ref{alg:powell}, we use at most $2$ oracle calls.
Finally, by Theorem~\ref{thm:geom_correct_total_lin}, the maximum number of consecutive geometry correcting iterations is $\mathcal{O}(q \log q)$.
The result follows.
\end{proof}

% \ks{Maybe we should describe the algorithm and/or change Algorithm~\ref{alg:tr_sub} to include model improving steps}. 
% When $m_k$ is constructed using by constrained least squares over a $\Lambda$-poised $\mathcal{Y}_k$ (and arbitrary $\mathcal{Z}_k$) and if we take $\Lambda = 1 + \mathcal{O}(\frac{1}{q})$ then by Theorem~\ref{thm:Lambda_to_kappaeg_err}, we have $\kappa_{eg} = \mathcal{O}(\sqrt{q})$ and $\epsilon_g = \mathcal{O}(\sqrt{q}\frac{\epsilon_f}{\Delta_{\min}})$.
% Like in \cite{ChaudhryScheinberg2026ICM}, when $Q_k$ is drawn from Haar distribution with $q\geq 3$,  $\theta=\frac{243}{443}$ and $\kappa_g^2={\cal O}({\frac{n-q}{n}})$ 
% is dependent on  $n$. 
% Then $C_1', C_1'' = \mathcal{O}(\sqrt{\frac{1}{q}})$.
% To satisfy the conditions of Theorem~\ref{thm:tr_sub_complexity}, we must take $\epsilon \geq \mathcal{O}( \max \{\sqrt{n \epsilon_f} + \sqrt{n} \frac{\epsilon_f}{\Delta_{\min}} , \sqrt{n} (\Delta_{\min} + \frac{\epsilon_f}{\Delta_{\min}}) \} )$.
% In this regime we recover the $\mathcal{O}(\frac{n}{\epsilon^2})$ iteration complexity.
% Since one can interpolate a poised set with $q+1$ oracle calls, this leads to a worst case expected oracle complexity of 
% \[
% \mathcal{O}\left(\frac{nq}{\epsilon^2}\right).
% \]

\section{Numerical Implementations and Results} 
\label{sec:numex}
In this section we propose an implementation  of Algorithm~\ref{alg:powell}, which incorporates all its elements such as the  self-correcting and geometry-correcting steps but in addition includes several practical features. Some of these features are borrowed from Powell's algorithms and some are new. As we will discuss, all these additional features improve practical performance but make the analysis more cumbersome. However, ultimately the order of the worst-case complexity of the algorithm is preserved.

The practical implementation is given in Algorithm~\ref{alg:gcdfo_rho}. It utilizes the two interpolation sets $\mathcal Y$ and $\mathcal Z$ to manage linear Lagrange polynomials while fitting quadratic models, as proposed in Algorithm~\ref{alg:powell}. We compare our proposed algorithm to NEWUOA \cite{MJDPowell_2004b}, which is arguably the most scalable of Powell's algorithms and which maintains geometry of the full interpolation set by the use of quadratic Lagrange polynomials, and with DFOTR \cite{DFOTRpaper} which is a surprisingly efficient method that does not maintain any Lagrange polynomials and only updates the sample set based on the distance of the points to the TR center. In that respect DFOTR also maintains two separate sets, in that it does not reduce the trust region on steps that are not successful and when there are fewer than $n+1$ points in the appropriate vicinity of the trust region center.

We also test a variation of Algorithm~\ref{alg:powell}, to which current theory does not extend and which uses the quadratic Lagrange polynomials combined with a NEWUOA-like self-correcting rule in place of the rule described in  Algorithm~\ref{alg:powell}. As this method seems to provide improvement in high-dimensional setting, it gives motivation for the theory from Sections \ref{sec:lagrange} and \ref{sec:powell} to be extended to quadratic Lagrange polynomials in future work.  All solvers will use a novel, adaptive fitting scheme. We test these algorithms on a collection of unconstrained problems from the CUTEst test set \cite{gratton2024s2mpjcutestoptimizationproblems} and investigate results in low dimension, high dimension, and in randomized subspaces.

The following Section~\ref{subsec:geo-correction-algorithm} describes the design choices of our proposed algorithms and how it still satisfies the theory. In Section~\ref{subsec:results} we discuss the testing methodology and numerical results.

\subsection{Practical Implementations} \label{subsec:geo-correction-algorithm}
Algorithm~\ref{alg:gcdfo_rho} (GC-YZ-LIN) is a practical implementation of the geometry correction framework which maintains linear Lagrange polynomials for the set $\mathcal Y$ while interpolating a quadratic model using $\mathcal Y \cup \mathcal Z \cup \{x_k\}$. 
This method can be seen as the middle ground between DFOTR and NEWUOA, where the former makes only minimal effort to ensure good sample set geometry while the latter uses a significant amount of effort.

We now describe the changes implemented in Algorithm~\ref{alg:gcdfo_rho} as opposed to Algorithm~\ref{alg:powell}.
Algorithm~\ref{alg:gcdfo_rho} makes use of a resolution floor and small step gate, both of which are ideas borrowed from Powell's methods.

{\bf The resolution floor $\sigma_k$} is an adaptive lower bound on $\Delta_k$ which only gets decreased once no progress can be made for that resolution. That is, if $\Delta_k = \sigma_k$, the geometry is good, and the iteration is still unsuccessful we reset $\sigma_{k+1} = \theta\sigma_k$ for some $0 <\theta < 1$. We then allow $\Delta_k$ to shrink even on model improving iterations as long as  $\Delta_k$ is larger than the floor $\sigma_k$. This is motivated by the observation that in higher dimensions it can be expensive to always ensure good geometry before shrinking, hence an alternative is to only ensure this good geometry at intervals throughout the trajectory of the algorithm.

%{\bf Small step gate} is used in addition to the small model gradient gate $\|g_k\|\geq \eta_2 \Delta_k$ by  including another  condition $\|s_k\|\geq c_g \sigma_k$ under which iteration may be deemed successful.  This change improves the termination speed of the algorithm. 
{\bf A small step gate and small model gradient gate} are used by checking the conditions $\|s_k\|\geq c_g \sigma_k$ and $\|g_k\|\geq \eta_2 \Delta_k$ respectively. The small step gate adds an additional condition under which an iteration is deemed successful, which we note improves the termination speed of the algorithm. In the small model gradient gate, before evaluating a proposed step, we check if the norm of the model gradient is small relative to the trust region radius. That is if we know the step is going to be rejected, we save a function evaluation by skipping the evaluation of the trial step and attempt a geometry correcting step and possibly shrink the trust region radius. The same principle applies for the small step gate.

With these two changes, the complexity analysis in Section \ref{sec:complexity}  
 follows through with small modifications, provided that the following simple assumption holds. 
\begin{assumption}\label{assum:step-length}
On every iteration $k$, we have
\[
\qquad \|s_k\|\;\ge\;\kappa_{\mathrm{step}}
\min\left\{\Delta_k,\frac{\|g_k\|}{\kappa_{bhm}}\right\}.
\]
\end{assumption}

Note that if the  trust region subproblem is solved exactly, the assumption is satisfied with $\kappa_{\mathrm{step}}=1$, while for the  Cauchy step we have $\kappa_{\mathrm{step}} = \frac{\kappa_{fcd}}{3}$. Then, for Lemma~\ref{lem:tr_success} to go through, we require that $c_g \leq \kappa_{\mathrm{step}}$ and Assumption~\ref{assum:step-length} to hold. It would then follow that 
\begin{align*}
    \kappa_{bhm} \Delta_k \leq \max\{\kappa_{bhm}, \eta_2\} \Delta_k \leq \|g_k\| \implies \min\left\{\Delta_k, \frac{\|g_k\|}{\kappa_{bhm}}\right\} = \Delta_k,
\end{align*}
hence $\|s_k\| \geq \kappa_{\mathrm{step}} \Delta_k \geq c_g \sigma_k$, and a small $\Delta_k$ still implies a successful iteration, even with a small step gate. The rest of the complexity argument for the full-dimensional method follows very closely to the analysis  in Section~\ref{sec:complexity}, by applying key results with $\sigma_k$ instead of $\Delta_k$ and deriving a bound on the number of unsuccessful iterations until $\Delta_k$ reaches the floor $\sigma_k$ for each round of $\sigma_k$ reductions. Hence the proposed method will obey the theory, at the cost of some logarithmic factors. 

\textbf{The model fitting procedure} fits a quadratic model, which we call the  ``hedge'' model, using the following somewhat elaborate sequence of steps. On each iteration we compute the minimum Frobenius norm (MFN) model and the minimum change Frobenius norm (MCFN) model.  Both of these models are computed by solving 
\begin{equation}\label{eqn:MCFN}
    \min_{g,H} \|H_{prev} - H\|_F \quad \text{s.t.} \quad v^\top g + \frac{1}{2}v^\top H v = f(x_k + v) - f(x_k),\ v \in \mathcal {Y} \cup \mathcal {Z}.
\end{equation}
where the MFN model sets $H_{prev}=0$ and the MCFN model sets $H_{prev}$ to the previous iteration's MCFN Hessian. For both models, an exponentially weighted moving average of a relative error score is maintained and is used to determine which model is to be used for the current iteration. Only one model is ``active'' at a time and if the error for the alternative model becomes lower than the active model, we switch which model is used to compute the trial step. This change comes from the observation that sometimes it may be beneficial to remember the current local curvature, while at other times it may be beneficial to fit a fresh model as previous model Hessians can become stale. The proposed scoring method is a way to adaptively decide which model is best suited for the current iteration. We remark that a similar idea was proposed in Powell's original paper, where consecutive poor steps would trigger an MFN model to be fit instead of MCFN.

After this modeling procedure, to ensure a bounded model Hessian, we check if the fitted Hessian exceeds the bound, i.e. if $\|H\| > K$. If so, we fit $g_k$ and $H_k$ by solving the regression problem~\eqref{eqn:model-regression}. We note that there are other ways to bound $\|H_k\|$, such as via clipping. Since several of the objective functions in our data set naturally have extreme curvature, to prevent the truncation of useful curvature information, we set $K$ to be a very large number. As a result, the condition $\|H\|> K$ will rarely trigger, which will yield either one of the MFN or MCFN models throughout an overwhelming majority of the iterations. We note that setting $K$ very large implies that for some problems $\kappa_{bhm}$ becomes very large. However, it appears to happen only when $L$ is similarly large, thus having large $\kappa_{bhm}$ or reducing it has no bearing on the order of the theoretical complexity bound.

We also develop a version of Algorithm~\ref{alg:gcdfo_rho}, which we call GC-YZ-V, where the self-correction step computes and updates quadratic Lagrange polynomials over the set $\mathcal Y \cup \mathcal Z$, similar to NEWUOA, while still maintaining the $\mathcal Y$-$\mathcal Z$ separation. Specifically, for the self-correcting step, we compute the score of a point in the sample set to be:
$${\rm score}(y_i) = \max{\left(1, \frac{\|y_i\|^2}{\max(0.1 \Delta_k, \sigma_k)^2}\right)}^3 \cdot \ell_i(s_k)^2,$$
and replace the point that achieved the maximum score by $s_k$. Intuitively, this attempts to simultaneously remove points that are deemed far away while attempting to improve geometry. After such self-correction, the geometry correction is carried out in the same manner as in GC-YZ-LIN. Since the pseudocode closely mirrors Algorithm~\ref{alg:gcdfo_rho} and is not currently accompanied by any theoretical guarantees, we omit it and present the results as a proof of concept.  We see that this algorithm somewhat outperforms both GC-YZ-LIN and NEWUOA and believe that modifications to only the self-correction step  can be ultimately covered by extending our theory.

\textbf{The random subspace variation with geometry correction} has a straightforward implementation, in that there is little deviation from the theoretical framework described in Algorithm~\ref{alg:powell_sub}. We call this algorithm GC-sub. The only deviation from the framework is the addition of a successful and unsuccessful iteration ``patience'' parameter. This value dictates the required number of successful or unsuccessful iterations before a subspace is redrawn.  Intuitively, if we are making significant progress on a particular subspace, it may be beneficial to keep this subspace instead of redrawing. In our theoretical framework Algorithm~\ref{alg:powell_sub} this parameter is set to $1$. While one can extend the theory to any  fixed value of this parameter, it would complicate the notation of Section \ref{sec:subspace_ffd}. On the other hand our computational results suggest that $1$ is the best value for this parameter at least in the current setting.

\begin{algorithm}[H]
    \caption{~\textbf{Geometry-correcting algorithm with $\mathcal Y$ and $\mathcal Z$ separation (GC-YZ-LIN)}}
    \label{alg:gcdfo_rho}
    {\bf Inputs:} A zeroth-order oracle $f(x)\approx {\phi}(x)$, $\Delta_0$,  $x_0$,  $\gamma \in (0,1)$, $\eta_1 >0$, $\eta_2 >0$, $\Lambda>1$, $\Lambda_{sc} \geq 1$, $\theta\in(0,\gamma)$, $\sigma_{\mathrm{end}}\in(0,\Delta_0], c_g \in (0,1]$, model Hessian bound $K$, averaging weight $\beta \in (0,1]$.\\
    {\bf Initialization:} An initial set ${\cal Y}_0$ such that $|{\cal Y}_0| =n$, an initial set ${\cal Z}_0$ such that $|{\cal Z}_0| \leq n(n+1)/2$ and the function values $f(x_0)$, $f(x_0 + y_i)$, $y_i \in {\cal Y}_0$, $f(x_0 + z_i)$, $z_i \in {\cal Z}_0$. Set the resolution floor $\sigma_0 = \Delta_0$. A set of Lagrange polynomials  
$\{\ell_i(x), i=1, \ldots, n\}$ in  ${\cal P}$ for the set ${\cal Y}_0$, model errors $e_C, e_F \gets \text{undefined}$, $H_{prev} \gets \mathbf 0$, set active model label $\alpha \gets {F}$.\\
\For{$k=0,1,2,\dots$} {
    \nl {\em Model building:} Construct MFN model $g_F, H_F$. Construct MCFN model $g_C, H_C$. Set $H_{prev} = H_C$.
    
    \nl  If $e_C$ or $e_F$ is undefined, then $\alpha \gets F$ and skip {\em model switching} step in Line~\ref{line:model-switching}.  
    
    \nl \label{line:model-switching}{\em Model switching:} If $\alpha = F$ and $e_C < 0.8 e_F$ then  $\alpha \gets  C$, otherwise if $\alpha = C$ and $e_F < 0.8 e_C$ then $\alpha \gets F$.

    \nl If $\alpha = F$ then $g_k, H_k \gets g_F, H_F$. If $\alpha = C$ then $g_k, H_k \gets g_C, H_C$.  
    
    \nl {\em Ensuring bounded model Hessians: } If $\|H_k\| > K$ set $g_k, H_k$ by solving $\eqref{eqn:model-regression}$.

    \nl Construct the set of Lagrange polynomials
        $\{\ell_i(x), i=1, \ldots, n\}$ in  ${\cal P}$ for the set ${\cal Y}_k$ and  find max  value
        \[
            {(j_k^*,s_k^*) =\arg \max_{j=1,\ldots, n,s\in B(0,\Delta_k)} |\ell_j(s)|.}
        \]\\

        % \nl \textcolor{blue}{If $\|g_k\| < \eta_2 \Delta_k$ \textbf{ and }
        % $|{\cal Y}| < n$, then ${\cal Y} \gets {\cal Y} \cup \{s^*_k\}$. Go to 1 and $k \gets k + 1$.} \\

    \nl \label{line:small-model-grad}{\em Small model gradient gate: } If $\|g_k\| < \eta_2 \Delta_k$ perform the following checks: 
    \begin{description}
        \item[] If $|\ell_{j^*_k}(s^*_k)|>\Lambda$ then perform ``\emph{Geometry correction of $\mathcal Y$ by replacing a ``bad'' point}''. 
       \item[] \label{alg:resolution-update} If 
       $|\ell_{j^*_k}(s^*_k)|\leq \Lambda$ and $\Delta_k \leq \sigma_k$ perform the {\em resolution update}: $
           \sigma_{k+1} = \max\{\theta\sigma_k,\ \sigma_{\mathrm{end}}\},\
           \Delta_{k+1} = \max\{\tfrac12\Delta_k,\ \sigma_{k+1}\}.
       $ 
       \item [] If 
       $|\ell_{j^*_k}(s^*_k)|\leq \Lambda$ and $\Delta_k > \sigma_k$, then $\Delta_{k+1} \gets \max\{\gamma \Delta_k,\ \sigma_k\}$.
       \item [] Go to 1 and $k \gets k + 1$. 
    \end{description}
    
    \nl Compute a trial step $s_k$ as in Algorithm~\ref{alg:tr}. \\

    \nl {\em Small step gate:} If $\|s_k\| < c_g\sigma_k$, do not evaluate $f(x_k+s_k)$. Perform the same checks as {\em small model gradient gate}.

    \nl Evaluate $f(x_k + s_k)$ and compute the ratio $\rho_k$  as in Algorithm~\ref{alg:tr}. \\
     \nl {\em Model scoring: } \For{$\mu\in\{\rm F,C\}$}{
        $p_\mu\leftarrow (g_{\mu})^\top s_k+\tfrac12 {s_k}^\top H_{\mu}{s_k}$, \ $f_{c} \gets f(x_k + s_k) - f(x_k)$ \\
        $\varepsilon_\mu\leftarrow
          \dfrac{|p_\mu- f_{c}|}{\max\{|f_{c}|,\,|p_\mu|\}}$
          
          If $e_\mu$ is undefined, then $e_\mu = \varepsilon_{\mu}$, otherwise $e_\mu = \beta e_\mu+(1-\beta)\varepsilon_\mu$.
    }
    \nl  {\em Successful iteration:} $\rho_k\geq \eta_1$. Perform the successful update as in Algorithm~\ref{alg:powell}.\\ 
    \nl{\em Unsuccessful or Model Improving iteration:} $\rho_k < \eta_1$. Perform all applicable steps $\bf{(i), \dots, (iv)}$ from Algorithm~\ref{alg:powell}.
    \begin{description} 
     \item[] {\em Model improving iteration:} If $I_{imp}=1$, set $\Delta_{k+1} = \max\{\gamma \Delta_k, \sigma_k\}$. 
    \item[] {\em Unsuccessful iteration above the floor:} If $I_{imp}=0$ and $\Delta_k > \sigma_k$, set $\Delta_{k+1} = \max\{\gamma \Delta_k, \sigma_k\}$.
    \item[] {\em Unsuccessful iteration at floor:} If $I_{imp}=0$ and $\Delta_k \leq \sigma_k$, perform the \emph{resolution update} from Line~\ref{line:small-model-grad}. 
    \end{description}
    }
\end{algorithm}

\subsection{Testing and Results} \label{subsec:results} 
We now turn to the results. We state upfront that our numerical comparison yields the following observations. 
\begin{enumerate}
    \item The new adaptive Hessian fitting procedure improves performance for all solvers.
    \item Methods with geometry correction, i.e. GC-YZ-LIN/V and NEWUOA, outperform DFOTR in higher dimensions. 
    \item Using the $\mathcal Y$-$\mathcal Z$ separation, GC-YZ-LIN is able to largely match NEWUOA, with trade-offs between early-game and late-game performance.
    Transplanting NEWUOA's self-correcting rule into the $\mathcal Y$-$\mathcal Z$ framework, we obtain a variant that outperforms NEWUOA; no theoretical analysis has been developed for it yet.
    \item A practical implementation of a randomized subspace algorithm that is competitive  with full space solvers remains difficult.
\end{enumerate}

Here we give the  details of our testing methodology. Table~\ref{tab:solver-params} outlines the parameters used for our proposed solver and DFOTR. We will use PRIMA's \cite{prima} implementation of NEWUOA, which has fixed trust region parameters by default, so we omit restating it in the table. All problems are initialized in the same way, where the initial interpolation set ${\mathcal Y}_0$ is set to $\{ \Delta_0 u_i\}$  where $u_i, i=1, \ldots, n$ are randomly rotated coordinate vectors, and radius
$\Delta_0 = 0.5$.  ${\mathcal Z}_0$ is then set to $\{-\Delta_0 u_i\}$. To allow the algorithm to run until an exhausted budget, we set the termination criteria to be $\Delta_{\rm end} = 10^{-12}$ for DFOTR and $\rho_{\rm end}, \sigma_{\rm end} = 10^{-12}$ for NEWUOA and GC-YZ-LIN/V respectively. All trust region subproblems are solved exactly. We compare performances using data profiles at tolerances $\tau \in \{10^{-2}, 10^{-4}, 10^{-6}\}$ and report the area under curve score of each solver in the legend. We average the curves obtained from the different randomly rotated initial sets.

We remark on the choice of three parameters in GC-YZ-LIN that also show up in the complexity analysis. Firstly, the model Hessian bound $K$ is set to be extremely large, namely $10^{100}$ and essentially inactive. This is largely a byproduct of testing solver performance on problems from CUTEst with naturally extreme curvature. For example, problems such as ``SSBRYND'' and ``SCURLY10'' have true Hessian norms greater than $10^{20}$ while ``POWERSUM'' can have a true Hessian norm even exceeding $10^{100}$. Thus setting a universal, fixed, and small $K$ will handicap the model's ability to capture the function's true curvature. Additionally,  we set $\eta_2$ to be small, namely $5\times 10^{-9}$. While ensuring the norm of the model gradient is not too small relative to the radius is necessary for the theory, a large $\eta_2$ tethering the trust region radius to the model gradient can be detrimental to problems with small local Lipschitz constants, for example, the Rosenbrock function outside of the ``banana valley''. If there is still progress to be made, iterates in this region will be forced to have a small trust region radius and may stunt progress. Moreover, in practice it appears beneficial to have a relaxed choice for $\Lambda$. According to our theory, $\Lambda$-poisedness guarantees a certain amount of function decrease on successful iterations each time. However, ensuring a small $\Lambda$ in practice comes at the cost of more function evaluations for geometry correction. We observe that the best performance comes when $\Lambda = 1000$, which suggests that although the geometry may not be perfect, the algorithm is still able to progress, and geometry correction should only take place when the geometry is exceedingly bad. We also make the distinction between $\Lambda$ and $\Lambda_{sc}$, where the latter is a threshold on how large a Lagrange polynomial value at a step would need to be to initiate a self-correcting step. In practice we set is to $2$,  while the analysis applies as long as $\Lambda_{sc} \geq 1$.

These seemingly detrimental parameter choices for $\eta_2$ and $\Lambda$ in terms of the complexity bounds can be explained by the fact that our analysis addresses the worst-case. Setting $\eta_2=\sqrt{n}$ and $\Lambda={\Theta}(1+\frac{1}{n})$ essentially forces more model improving iterations before a successful step is allowed or trust region radius is reduced, to ensure  best models and progress according to the worst case $\kappa_{eg}$ constant. However, in practice the error between the model and the function at the trial step can be much smaller than the worst case $\kappa_{eg}$ bound suggests, so too many model improvement steps may be wasteful. It would be interesting to explore adaptive choices for these constants in future work.

\begin{table}[htpb]
\centering
\caption{Solver settings used throughout testing.}
\label{tab:solver-params}
\begin{tabular}{llll}
\toprule
 & GC-YZ-LIN/V & GC-sub & DFOTR \\
 \midrule
Acceptance ratio $\eta_1$       & 0.01 & 0.1  & 0.05 \\
Model gradient threshold $\eta_2$        & $5 \times 10^{-9}$ & $5 \times 10^{-9}$  & ---  \\
``Very successful iteration'' threshold     & --- & ---  & 0.5  \\
Expansion factor $\gamma_{\mathrm{inc}}$ & 1.3 & 1.3 & 1.6 \\
Contraction factor $\gamma_{\mathrm{dec}}$     & 0.8  & 0.8 & 0.8  \\
Resolution shrink $\theta$      & 0.1  & ---  & --- \\
Poisedness threshold $\Lambda$  & 1000 & 1000 & --- \\
Self-correcting threshold $\Lambda_{\mathrm{sc}}$ & 2 & 2 & --- \\
Hedge: EWMA decay $\beta$       & 0.8  & ---  & --- \\
Model Hessian bound $K$         & $10^{100}$ & --- & --- \\
Small step gate threshold $c_g$         & 0.5 & --- & --- \\
\midrule
Subspace dimension $q$          & ---  & 5    & --- \\
Interpolation set capacity  & ---  & $2q$ & --- \\
\bottomrule
\end{tabular}
\end{table}

\textbf{Low-dimensional tests.} We begin our testing on 174 low-dimensional test problems with dimension between $2\leq n \leq 5$. Since the dimension is low, we set the maximum allowable number of points in the sample set to be $(n+1)(n+2)/2$ points. Aside from the first few iterations, the sample set of GC-YZ-LIN/V and NEWUOA will eventually become full, resulting in an interpolation system that is fully determined, so the two models coincide and the switching is inactive. In contrast, DFOTR often removes several sample points at once, so the hedge model can be active throughout the trajectory.  Figure \ref{fig:low-dim} shows the performance of the solvers on this test set. All solvers seem to perform comparably well.

\begin{figure}
    \centering
    \includegraphics[width=1\linewidth]{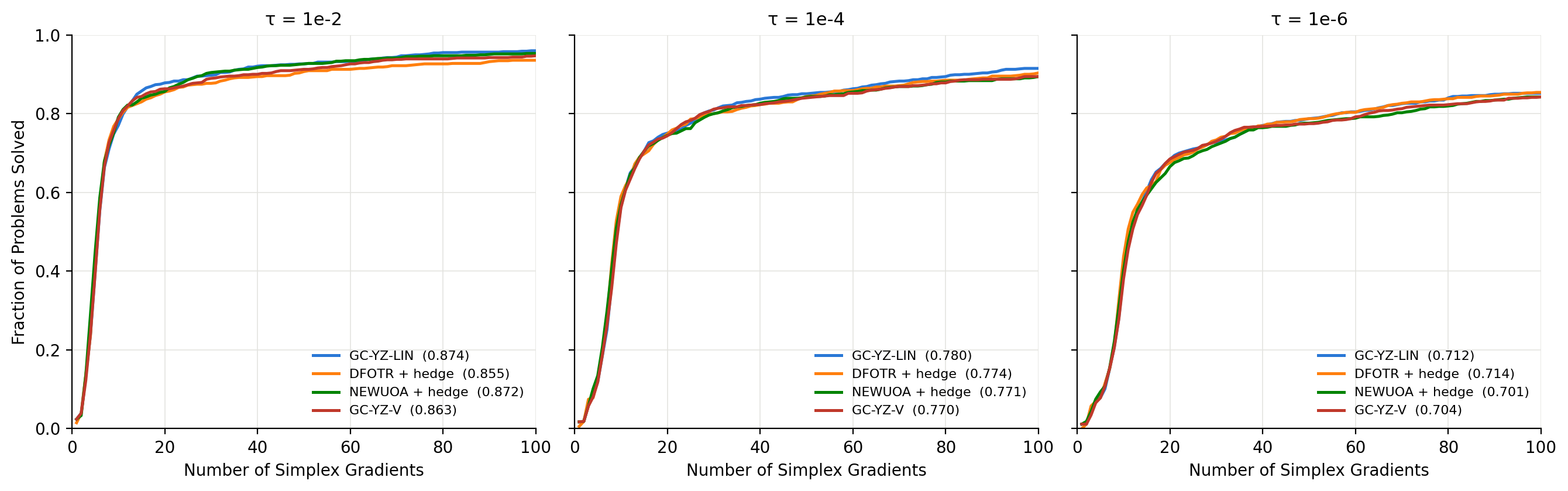}
    \caption{Data profiles for the low-dimensional test suite: $2 \leq n \leq 5$.}
    \label{fig:low-dim}
\end{figure}

\textbf{Medium-/High-dimensional tests.} We also benchmark on test sets in dimension $30$, $100$, and $200$, with $95$, $101$, and $92$ problems respectively in each set. For these tests, we will use a maximum of $2n+1$ sample points in the total interpolation set. Firstly, Figure~\ref{fig:hedge-model} compares the solvers using their default fitting routine against using the hedge model. We see that adopting this model strictly improves performance across the board. In what follows, to ensure fair testing, all solvers are benchmarked using the hedge model.

\begin{figure}
    \centering
    \includegraphics[width=1\linewidth]{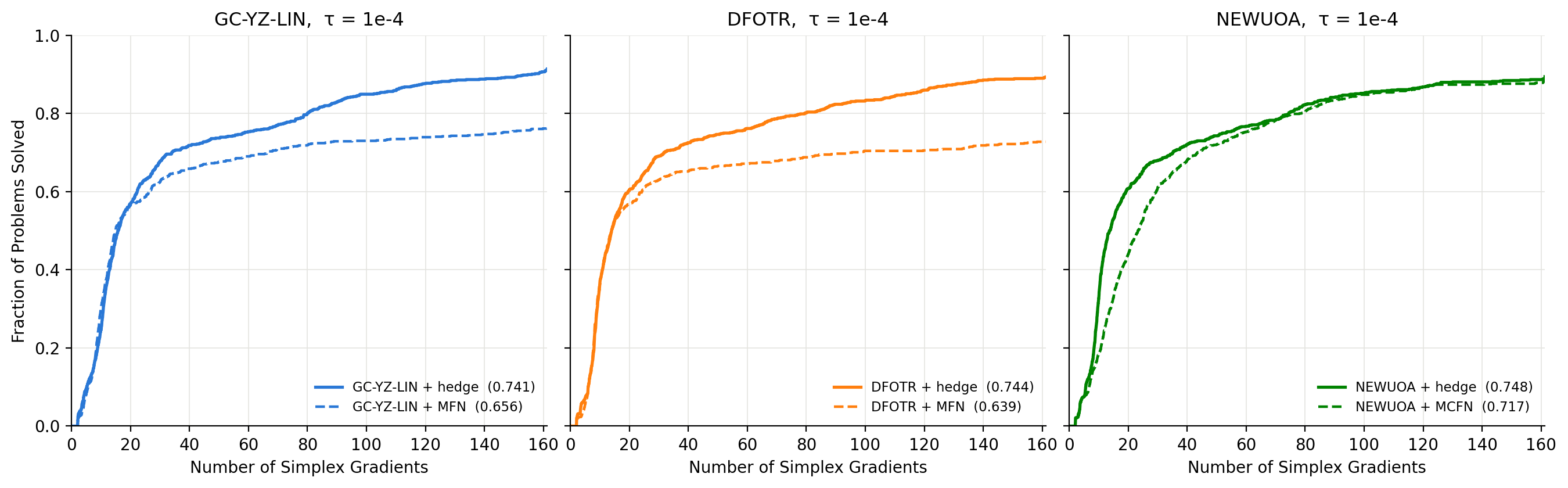}
    \caption{Data profiles comparing the hedge model with each solver's default fitting procedure. }
    \label{fig:hedge-model}
\end{figure}

Figures~\ref{fig:n=30-full},~\ref{fig:n=100-full}, and~\ref{fig:n=200-full} show the data profiles for the $30$-, $100$- and $200$- dimensional test sets respectively. The first most apparent observation is the deterioration of DFOTR as dimension increases. This is largely due to the absence of a geometry management procedure. As dimension increases, the $2n+1$ points become increasingly sparse and are more likely to be nearly degenerate. Thus, DFOTR performs increasingly aggressive shrinking to eventually force a geometry reset, which leads to radius collapse and premature termination, as seen in the significant flattening of the DFOTR data profile curves.

Another apparent observation is the almost strict improvement of GC-YZ-V over NEWUOA and GC-YZ-LIN. GC-YZ-V is a combination of several algorithmic features from both NEWUOA and GC-YZ-LIN. This result suggests that the use of NEWUOA-style self-correction may be more efficient than first ensuring all points are close, then checking the self-correction property. While the latter is more principled and easier to analyze theoretically, it may be slower in practice. This also suggests it may be worthwhile to maintain a separation of sets $\mathcal Y$-$\mathcal Z$, even when quadratic  Lagrange polynomials are maintained for the entire sample set $\mathcal Y \cup \mathcal Z$. Again, we note that this variant is not covered by current theory  and is a line of future work.

Finally, we observe that GC-YZ-LIN is closely competitive with NEWUOA across dimensions and across tolerances. At $\tau = 10^{-2}$, the curves are nearly identical. At $\tau = 10^{-4}, 10 ^{-6}$, a pattern emerges where NEWUOA is quicker to solve more problems early in the run, however GC-YZ-LIN eventually catches up and overtakes NEWUOA later in the run. In general, the performance of the two methods is comparable, which is surprising given that GC-YZ-LIN only manages linear geometry for a subset of the sample set compared to NEWUOA's full quadratic geometry.

\begin{figure}
    \centering
    \includegraphics[width=\linewidth]{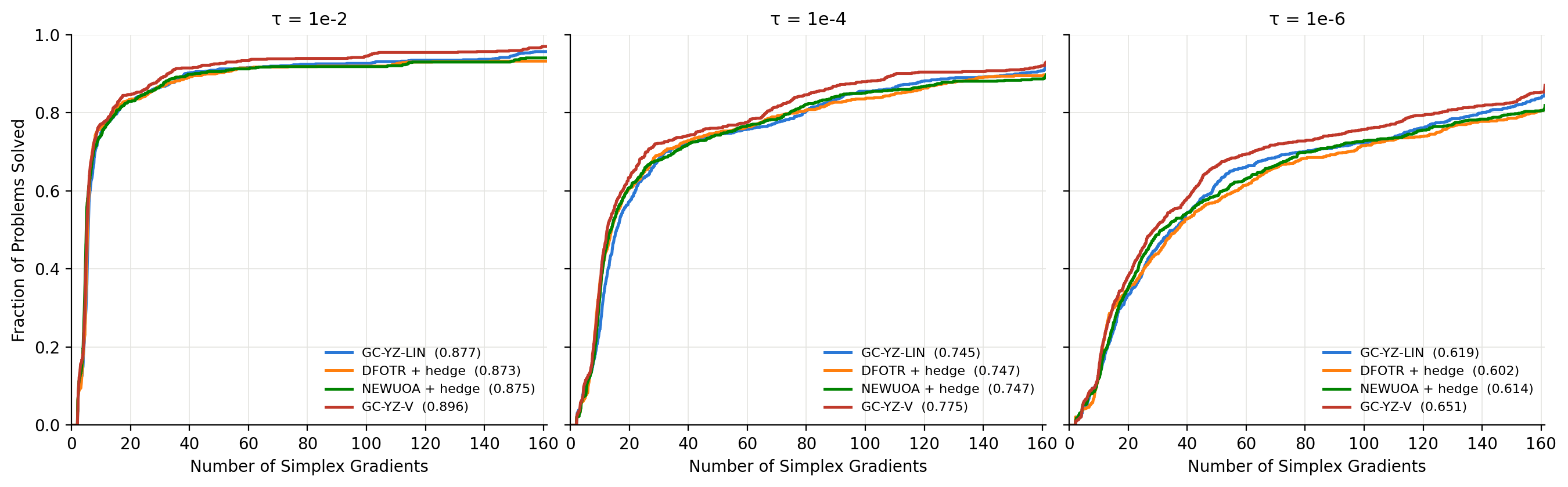}
    \caption{Data profiles for the test suite $n = 30$.}
    \label{fig:n=30-full}
\end{figure}
\begin{figure}
    \centering
    \includegraphics[width=\linewidth]{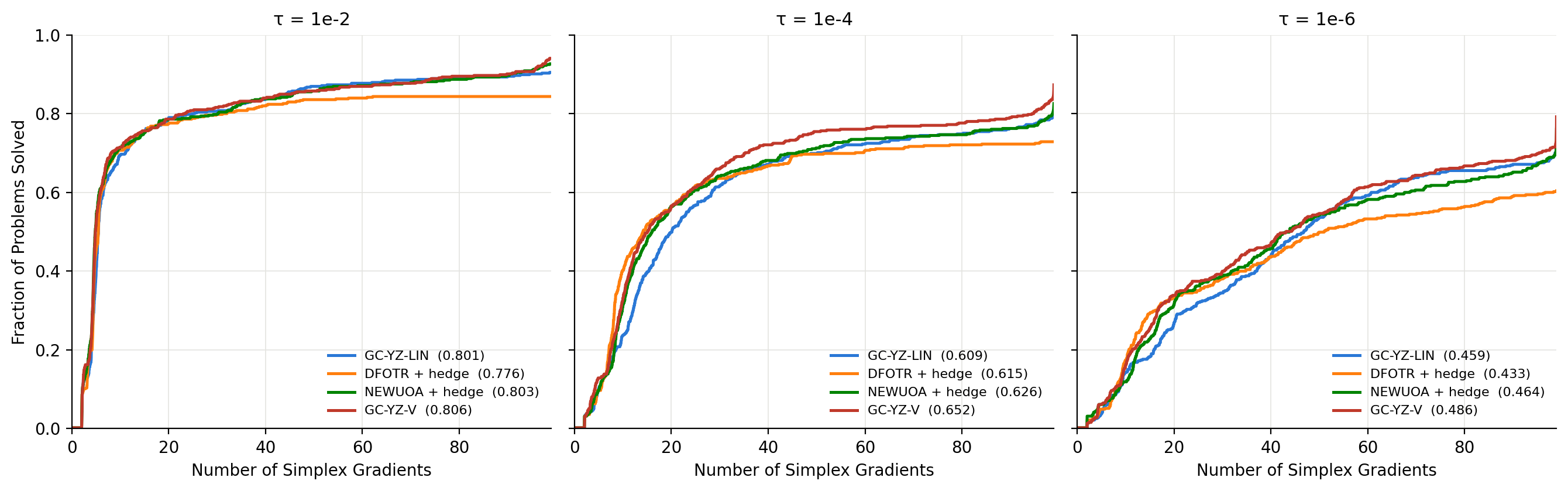}
    \caption{Data profiles for the test suite $n = 100$.}
    \label{fig:n=100-full}
\end{figure}
\begin{figure}
    \centering
    \includegraphics[width=1\linewidth]{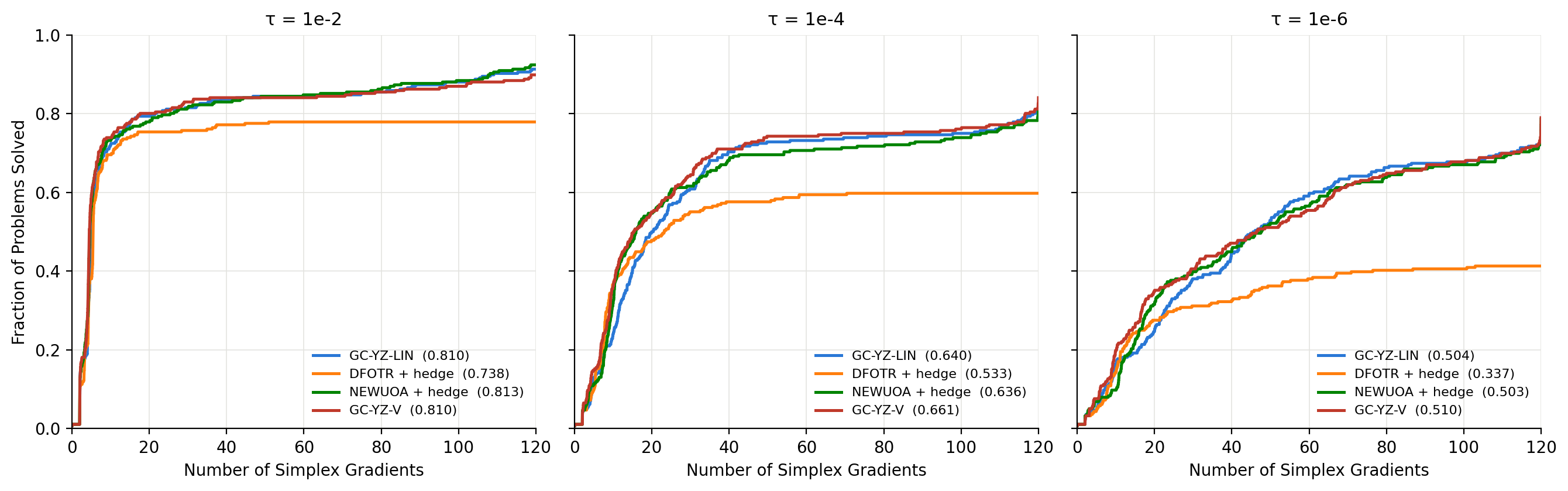}
    \caption{Data profiles for the test suite $n = 200$.}
    \label{fig:n=200-full}
\end{figure}

\textbf{Results in subspaces.} Figure~\ref{fig:subspace} compares the performance of the random subspace method against GC-YZ-LIN. The subspace method sets $q=5$, patience parameters equal to $1$, and draws a fresh coordinate stencil on every subspace redraw. A maximum of $2q+1$ points is used in the sample. We remark that varying choices of patience parameter, subspace dimension, re-draw strategy, and maximum number of allowed sample points were tested, and we opt to present the best performing combination found in dimension $30$ and $100$. It is evident that the subspace method does not perform as well as a full space solver and its performance deteriorates as tolerances get smaller. While theory suggests that as dimension increases, the competitiveness of subspace methods should as well, this is not something that is immediately evident.   A plausible explanation for the discrepancy between theory and practice is as follows. The random subspace method  needs at least ${\cal O}(q)$ points for each subspace redraw and the bound on the number of subspace redraws is not loose. Thus we arrive at the worst case complexity of order $q$ without fail. However, in the full space solvers, we rarely need more than a few consecutive geometry correcting steps to achieve good geometry, hence in practice we are typically far from actually attaining the worst-case behavior described in the complexity bound. Another explanation is that subspace methods ultimately lose the  advantage of quadratic approximation in the whole space, so may only be competitive where full space quadratic approximation is prohibitive or useless. 
 A future implementation which initiates any sample set with as few as two sample points may improve performance of subspace methods.
 In addition reusing sample points from other subspaces can be beneficial.  In conclusion, our contribution in terms of subspace TR method in this paper is mainly theoretical while practical approaches require further investigation. We would like to point the reader to the extensive empirical study carried out in \cite{Cartis2023}.

%\ks{What is (crown) in Figure 6.6? Why are we comparing to NEWUOA + hedge, when the subspace method does not use Hedge. What do we use for q? Do we use $2q+1$ sample points and why, given that $q$ is small? I do believe that the full description of Algorithm \ref{alg:gcdfo_ss} should be replaced with a few sentences explaining the key points. The rest is as before. }
%\ss{Added details above. Side comment: I think the difficulty in engineering a good performing randomized subspace method is as follows (and why the best patience param is =1 and (n+1)(n+2)/2 points may be detrimental). The algorithm has many opportunities to generate a subspaces. Some of them will not be well aligned and some of them will be well aligned. It is difficult know if it is well aligned. If it is well aligned, we'd want to stay on this subspace longer. However universally staying longer on all subspaces is costly, especially on poorly aligned subspaces. Another challenge is the management of the trust region radius, there's no guarantee that the ``resolution'' of a radius is the same from one subspace to the next. It could be the case the TR radius is too large and would require shrinking before taking advantage of its good alignment. This is in contrast to the constructed subspaces pay for the certainty of knowing which subspace is good, thus staying on it (and allowing the trust region to grow and shrink and run its course) is always good. Is this intuition correct?}
\begin{figure}
    \centering
    \includegraphics[width=\linewidth]{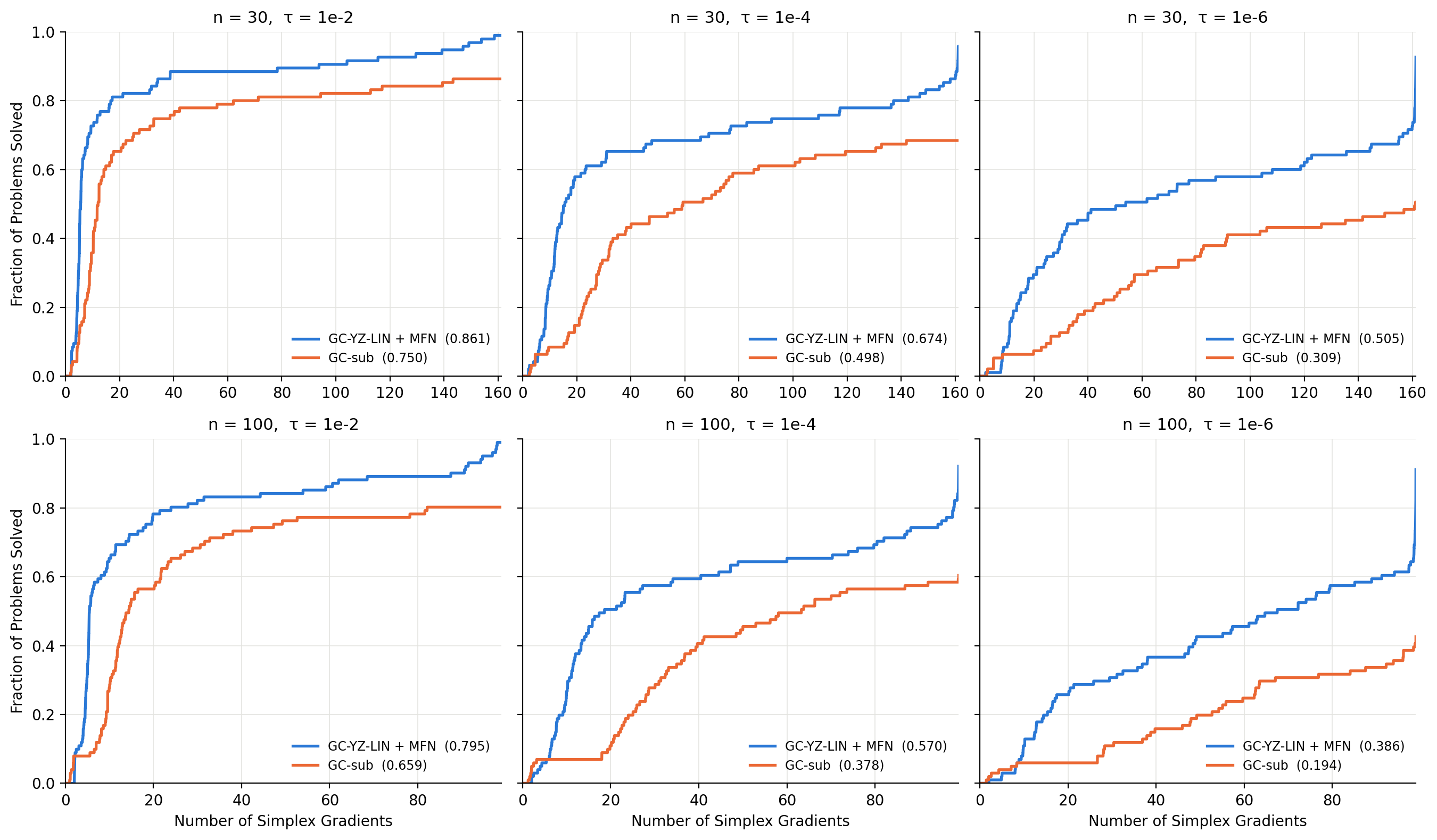}
    \caption{Data profiles for the random subspace implementation compared against GC-YZ-LIN for $n=30, 100$.}
    \label{fig:subspace}
\end{figure}

\subsection*{Acknowledgments}
\label{sec:Acknowledgements}
This work was partially supported by  ONR award N00014-22-1-215 and the Gary C. Butler Family Foundation.
%\input{acknowledgements}

%%%%%%%%%%%%%%%%%%%%%%%%%%%%%%%%%%%%%%%%%%%%%%%%%%%%
\bibliographystyle{plain}
\bibliography{references}

@book{TRbook,
  title={Trust region methods},
  author={Conn, Andrew R and Gould, Nicholas IM and Toint, Philippe L},
  year={2000},
  publisher={SIAM}
}

@book{audet2017derivativefree,
  author    = {Audet, Charles and Hare, Warren L.},
  title     = {Derivative-Free and Blackbox Optimization},
  series    = {Springer Series in Operations Research and Financial Engineering},
  publisher = {Springer International Publishing, Cham},
  year      = {2017},
  pages     = {302},
  isbn      = {978-3-319-68913-5},
  doi       = {10.1007/978-3-319-68913-5},
}

@article{rios2013derivativefree,
  author    = {Rios, Luis Miguel and Sahinidis, Nikolaos V.},
  title     = {Derivative‐free optimization: a review of algorithms and comparison of software implementations},
  journal   = {Journal of Global Optimization},
  year      = {2013},
  volume    = {56},
  number    = {3},
  pages     = {1247--1293},
  doi       = {10.1007/s10898-012-9951-y},
}

@article{UOBYQA,
  title={{UOBYQA}: unconstrained optimization by quadratic approximation},
  author={Powell, Michael JD},
  journal={Mathematical Programming},
  volume={92},
  number={3},
  pages={555--582},
  year={2002},
  publisher={Springer}
}

@book{DFObook,
  title={Introduction to derivative-free optimization},
  author={Conn, Andrew R and Scheinberg, Katya and Vicente, Lu\'{i}s N},
  year={2009},
  publisher={SIAM}
}

@article{DFOTRpaper,
  title={Computation of sparse low degree interpolating polynomials and their application to derivative-free optimization},
  author={Bandeira, Afonso S and Scheinberg, Katya and Vicente, Lu\'{i}s N},
  journal={Mathematical Programming},
  volume={134},
  number={1},
  pages={223--257},
  year={2012},
  publisher={Springer}
}

@article{bandeira2014convergence,
  title={Convergence of trust-region methods based on probabilistic models},
  author={Bandeira, Afonso S and Scheinberg, Katya and Vicente, Lu\'{i}s N},
  journal={SIAM Journal on Optimization},
  volume={24},
  number={3},
  pages={1238--1264},
  year={2014},
  publisher={SIAM}
}

@article{cartis2018global,
  title={Global convergence rate analysis of unconstrained optimization methods based on probabilistic models},
  author={Cartis, Coralia and Scheinberg, Katya},
  journal={Mathematical Programming},
  volume={169},
  number={2},
  pages={337--375},
  year={2018},
  publisher={Springer}
}

@article{blanchet2019convergence,
  title={Convergence rate analysis of a stochastic trust-region method via supermartingales},
  author={Blanchet, Jose and Cartis, Coralia and Menickelly, Matt and Scheinberg, Katya},
  journal={INFORMS Journal on Optimization},
  volume={1},
  number={2},
  pages={92--119},
  year={2019},
  publisher={INFORMS}
}

@article{gratton2018complexity,
  title={Complexity and global rates of trust-region methods based on probabilistic models},
  author={Gratton, Serge and Royer, Cl{\'e}ment W and Vicente, Lu{\'\i}s N and Zhang, Zaikun},
  journal={IMA Journal of Numerical Analysis},
  volume={38},
  number={3},
  pages={1579--1597},
  year={2018},
  publisher={Oxford University Press}
}

@article{berahas2021global,
  title={Global convergence rate analysis of a generic line search algorithm with noise},
  author={Berahas, Albert S and Cao, Liyuan and Scheinberg, Katya},
  journal={SIAM Journal on Optimization},
  volume={31},
  number={2},
  pages={1489--1518},
  year={2021},
  publisher={SIAM}
}

@article{berahas2021theoretical,
  title={A theoretical and empirical comparison of gradient approximations in derivative-free optimization},
  author={Berahas, Albert S and Cao, Liyuan and Choromanski, Krzysztof and Scheinberg, Katya},
  journal={Foundations of Computational Mathematics},
  pages={1--54},
  year={2021},
  publisher={Springer}
}

@article{jin2021high,
  title={High probability complexity bounds for line search based on stochastic oracles},
  author={Jin, Billy and Scheinberg, Katya and Xie, Miaolan},
  journal={Advances in Neural Information Processing Systems},
  volume={34},
  pages={9193--9203},
  year={2021}
}

@article{more2009benchmarking,
  title={Benchmarking derivative-free optimization algorithms},
  author={Mor{\'e}, Jorge J and Wild, Stefan M},
  journal={SIAM Journal on Optimization},
  volume={20},
  number={1},
  pages={172--191},
  year={2009},
  publisher={SIAM}
}

@article{powell2001lagrange,
  title={On the Lagrange functions of quadratic models that are defined by interpolation},
  author={Powell, MJD},
  journal={Optimization Methods and Software},
  volume={16},
  number={1-4},
  pages={289--309},
  year={2001},
  publisher={Taylor \& Francis}
}

@Article{Cartis2023,
author={Cartis, Coralia
and Roberts, Lindon},
title={Scalable subspace methods for derivative-free nonlinear least-squares optimization},
journal={Mathematical Programming},
year={2023},
month={May},
day={01},
volume={199},
number={1},
pages={461-524},
issn={1436-4646},
doi={10.1007/s10107-022-01836-1},
url={https://doi.org/10.1007/s10107-022-01836-1}
}

@article{scg,
author = {Scheinberg, K. and Toint, Ph. L.},
title = {Self-Correcting Geometry in Model-Based Algorithms for Derivative-Free Unconstrained Optimization},
journal = {SIAM Journal on Optimization},
volume = {20},
number = {6},
pages = {3512-3532},
year = {2010},
doi = {10.1137/090748536}
}

@techreport{MJDPowell_2004b,
  author = {Michael J D Powell},
  title = {The {NEWUOA} software for unconstrained optimization without derivatives},
  institution = {Department of Applied Mathematics and Theoretical Physics,
University of Cambridge},
  number = {DAMTP 2004/NA08},
  year = {2004}
}

@article{LarsMeniWild2019,
  title={Derivative-free optimization methods},
  author={Larson, Jeffrey and Menickelly, Matt and Wild, Stefan M},
  journal={Acta Numerica},
  volume={28},
  pages={287--404},
  year={2019},
  publisher={Cambridge University Press}
}

@article{cao2023first,
  author    = {Cao, Liyuan and Berahas, Albert S. and Scheinberg, Katya},
  title     = {First- and Second-Order High Probability Complexity Bounds for Trust-Region Methods with Noisy Oracles},
  journal   = {Mathematical Programming},
  year      = {2023},
  volume    = {207},
  number    = {1-2},
  pages     = {573--624},
  doi       = {10.1007/s10107-023-01999-5},
  eprint    = {arXiv:2205.03667},
  note      = {Also available as a preprint at arXiv, May 2022}
}

@inproceedings{ChaudhryScheinberg2026ICM,
  author    = {Abraar Chaudhry and Katya Scheinberg},
  title = {On Complexity of Model-Based Derivative-Free Methods},
  booktitle = {Proceedings of the International Congress of Mathematicians (ICM) 2026},
  address   = {Philadelphia, PA, USA},
  year      = {2026},
  publisher={SIAM},
  pages={208-228},
  volume={7}
}

@article{dzahini2024stochastic,
  author    = {Dzahini, Kwassi Joseph and Wild, Stefan M.},
  title     = {Stochastic Trust-Region Algorithm in Random Subspaces with Convergence and Expected Complexity Analyses},
  journal   = {SIAM Journal on Optimization},
  volume    = {34},
  number    = {3},
  pages     = {2743--2774},
  year      = {2024},
  doi       = {10.1137/22M1524072},
  url       = {https://doi.org}
}

@misc{cartis2026note,
      title={A note on the complexity of random subspace model-based methods for derivative-free optimization}, 
      author={Coralia Cartis and Lindon Roberts},
      year={2026},
      eprint={2608.17307},
      archivePrefix={arXiv},
      primaryClass={math.OC},
      url={https://arxiv.org/abs/2608.17307}, 
}

@misc{prima,
    title        = {{PRIMA: Reference Implementation for Powell's Methods with Modernization and Amelioration}},
    author       = {Zhang, Z.},
    howpublished = {available at http://www.libprima.net, DOI: 10.5281/zenodo.8052654},
    year         = {2023}
}

@article{pdfo,
    author       = {Ragonneau, T. M. and Zhang, Z.},
    title        = {{PDFO}: a cross-platform package for {P}owell's derivative-free optimization solvers},
    journal={Mathematical Programming Computation},
    volume={16},
    number={4},
    pages={535-559},
    year={2024},
    doi={10.1007/s12532-024-00257-9},
    url={https://doi.org/10.1007/s12532-024-00257-9}
}

@article{garmanjani2016trust,
  author  = {Garmanjani, R. and J{\'u}dice, D. and Vicente, L. N.},
  title   = {Trust-Region Methods Without Using Derivatives: Worst Case Complexity and the Non-Smooth Case},
  journal = {SIAM Journal on Optimization},
  volume  = {26},
  number  = {4},
  pages   = {1987--2011},
  year    = {2016},
  doi     = {10.1137/151005683},
  url     = {https://epubs.siam.org/doi/10.1137/151005683}
}

@misc{gratton2024s2mpjcutestoptimizationproblems,
      title={S2MPJ and CUTEst optimization problems for Matlab, Python and Julia}, 
      author={Serge Gratton and Philippe L. Toint},
      year={2024},
      eprint={2407.07812},
      archivePrefix={arXiv},
      primaryClass={math.OC},
      url={https://arxiv.org/abs/2407.07812}, 
}

\end{document}